\documentclass[a4paper,11pt]{article}
\usepackage[utf8]{inputenc}
\usepackage{geometry}

\usepackage[T5,T2A,T1]{fontenc}

\usepackage{listings}
\makeatletter
\def\@fnsymbol#1{\ensuremath{\ifcase#1\or \star\or\mathsection\or \mathparagraph\or \|\or **\or \dagger\dagger
   \or \ddagger\ddagger \else\@ctrerr\fi}}
\makeatother

\usepackage[russian, english]{babel}
\usepackage{amssymb}
\usepackage{mathtools}
\usepackage{faktor}
\usepackage{amsmath}
\usepackage{bbm}

\usepackage{url}

\usepackage{graphicx}
\usepackage{relsize}

\usepackage{paralist}
\usepackage{enumitem}

\usepackage{amsthm}

\newtheorem{thm}{Theorem }[section]
\newtheorem{prop}[thm]{Proposition }
\newtheorem{lem}[thm]{Lemma }
\newtheorem{cor}[thm]{Corollary }
\theoremstyle{definition}
\newtheorem*{definition}{Definition}

\usepackage{pst-node}

\usepackage{float}

\usepackage{csquotes}

\usepackage{comment}
\usepackage[bookmarks=true,bookmarksopen=true,colorlinks=true,linkcolor=blue,citecolor=blue,urlcolor=blue]{hyperref}
\usepackage{xcolor}
\usepackage{csquotes}
\usepackage{todonotes}

\usepackage{tikz}
\usetikzlibrary{cd}
\usetikzlibrary{babel}

\usepackage{hyperref}
\usepackage{cleveref}

\newcommand{\Frac}[0]{\text{\normalfont Frac}}

\newcommand{\hen}[1]{#1^{\text{\normalfont h}}}
\newcommand{\e}[1]{#1^{\text{\normalfont e}}}
\newcommand{\ps}[1]{[\![#1]\!]}
\newcommand{\aps}[1]{\langle #1 \rangle}
\newcommand*{\ord}{\text{\normalfont ord}}
\newcommand*{\idp}{\mathfrak{p}}
\newcommand*{\idq}{\mathfrak{q}}
\newcommand*{\idm}{\mathfrak{m}}
\newcommand*{\idn}{\mathfrak{n}}
\newcommand*{\ida}{\mathfrak{a}}

\newcommand*{\invlim}{\varprojlim}
\newcommand*{\dirlim}{\varinjlim}

\DeclareMathOperator{\GL}{GL}

\title{Diagonal Representation of Algebraic Power Series:\\
A Glimpse Behind the Scenes\footnote{This work is based on the author's master's thesis (U. of Vienna, 2020), supervised by H. Hauser.}}
\author{Sergey Yurkevich\thanks{Supported by the \href{https://www.fwf.ac.at/}{Austrian Science Fund} (FWF): P-31338.}}

\begin{document}
\maketitle

\vspace{-0.5cm}

\begin{abstract}
There are many viewpoints on algebraic power series, ranging from the abstract ring-theoretic notion of Henselization to the very explicit perspective as diagonals of certain rational functions. To be more explicit on the latter, Denef and Lipshitz proved in 1987 that any algebraic power series in $n$ variables can be written as a diagonal of a rational power series in one variable more. Their proof uses a lot of involved theory and machinery which remains hidden to the reader in the original article. In the present work we shall take a glimpse on these tools by motivating while defining them and reproving most of their interesting parts. Moreover, in the last section we provide a new significant improvement on the Artin-Mazur lemma, proving the existence of a 2-dimensional code of algebraic power series.
\end{abstract}

\noindent

\section{Introduction} \label{sec:1}
\subsection{Basic Notions and Motivation}
In all this text, $K$ will denote a field of characteristic zero, even though most presented results are known to work even for excellent local integral domains. By default, $\mathbb{N},\mathbb{Q}, \mathbb{R}$ and $\mathbb{C}$ are sets (equipped with the appropriate algebraic structure) of natural numbers (including $0$), rationals, reals and complex numbers respectively. A ring is always commutative with $1$ and a ring homomorphism is always unital. By $R^*$ we denote the set of units of a ring $R$, and $\widehat{R}$ stands for the algebraic completion of a local ring with respect to its maximal ideal. If not indicated otherwise, $x = (x_1,\dots, x_n)$ is a vector of $n$ variables and $x' = (x_1,\dots, x_{n-1})$. In contrast,~$t$ is always one variable and when we write $xt$, we mean $(x_1t,\dots,x_nt)$. Given an $n$-dimensional index $\alpha\in \mathbb{N}^n$ we write $|\alpha|$ for $\alpha_1+\cdots+\alpha_n$ and $x^\alpha$ for $x_1^{\alpha_1}\cdots x_n^{\alpha_n}$.

\begin{definition}
A formal power series $h(x) \in K\ps{x}$ is called \emph{algebraic} if there exists a non-zero polynomial $P(x,t) \in K[x,t]$ such that $P(x,h(x)) = 0$. Such a polynomial with minimal degree in $t$ is called a \emph{minimal polynomial of $h(x)$}. The set of algebraic power series is denoted by $K\langle x \rangle$. A power series which is not algebraic is called \emph{transcendental}.
\end{definition}
\noindent
Consider $(x) \subseteq K[x]$, the maximal ideal inside $K[x]$ generated by the elements $x_1,\dots,x_n$. Then, algebraically speaking, $K\aps{x}$ is the algebraic closure of the localization of $K[x]$ with respect to this ideal, $K[x]_{(x)}$, inside $K\ps{x}$. Hence the ring of algebraic power series is a subring of formal power series. Moreover it is easy to see that $K\aps{x}^* = K\aps{x} \cap K\ps{x}^*$. 
\noindent
Here are several examples in one variable, $x= x_1$.
{\setlength{\parindent}{0pt}
\begin{enumerate}[label=\textit{Example \arabic*:},wide]
    %{\it Example }\em 1., 
    \item Any polynomial $p(x) \in K[x]$ is an algebraic power series since we may chose $P(x,t) \coloneqq t - p(x)$.
	\item The power series given by 
	\[
		(1+x)^r = \sum_{k \geq 0} \binom{r}{k} x^k,
	\]
	for some rational number $r \in \mathbb{Q}$ is algebraic\footnote{The function $(1+x)^r$ exists for any $r\in \mathbb{Q}$, because the characteristic of $K$ is assumed to be zero.}. This holds true, because when $r = p/q$ for non-zero integers $p,q$, we may choose $P(x,t) = t^q - (1 + x)^p$ if $p,q>0$ and $P(x,t) = t^q (1 + x)^{-p}-1$ if $p$ happens to be negative. We obtain again that $P(x,(1+x)^r) = 0$.
	\item Consider the exponential power series: 
	\[
	\exp(x) = \sum_{k\geq 0} \frac{x^k}{k!}.\] 
	We claim that it is transcendental: assume it was algebraic, then after dividing the minimal polynomial by the coefficient of the leading term in $t$, we would find a $Q(x,t) = q_0(x) + \cdots + q_{m-1}(x)t^{m-1} + t^m \in K(x)[t]$ with $Q(x,\exp(x)) =0$. Now, taking the derivative of $Q(x,\exp(x)) =0$ with respect to $x$ and using $\exp'(x)=\exp(x)$, it follows that
	\[
		q_0'(x) + \dots + (q_{m-1}'(x)+(m-1)q_{m-1}(x))\exp(x)^{m-1} + m\exp(x)^m = 0.
	\]
	Subtracting this from $mQ(x,\exp(x))=0$ and using the simple fact that no rational function $q(x)$ can satisfy $q(x) = c q'(x)$ for $c \in K^*$, we can find a non-zero polynomial of lower degree than $m$ also annihilating $\exp(x)$. This is a contradiction with the minimality of $m$.
	\item The function $f(x) = \sqrt{x}$ is not an algebraic power series, because it is not a formal power series. 
	\item Set $f(x) = \sqrt{x+1}$ and $g(x) = \sqrt[3]{x+1}$; we already saw in Example $2$ that both $f,g \in K\aps{x}$. To see that $f(x)+g(x) = \sqrt{x+1} + \sqrt[3]{x+1}$ is algebraic as well, just consider the polynomial 
	\begin{align*}
		P(x,t) =	{t}^{6}-  3\left( x+1 \right) {t}^{4}- 2\left( x+1 \right) {t}^{3}+3\left( x+1 \right) ^{2}{t}^{2}-6 \left( x+1 \right) ^{2}t-x \left( x+1 \right) ^{2}
	% P(x,t) = -x^3-& 2x^2-x+t^6- 3xt^4-3t^4-2xt^3-2t^3\\
	% & +3x^2t^2-2t^3+3x^2t^2+6xt^2+3t^2-6x^2t-12xt-6t
	\end{align*}
	and verify that it indeed satisfies $P(x,f(x)+g(x)) =0$. Finding such a $P(x,t)$ is not straightforward and may require some work.
\end{enumerate}
}
Algebraic power series appear in many mathematical areas, such as combinatorics, algebraic geometry and number theory. In order to motivate their deep ring-theoretic study in the next sections, we first introduce a very explicit viewpoint.
\begin{definition}
	Let $g(x), f(x,t)$ be formal power series:
	\begin{align*}
	g(x) &= \sum_{i_1,\dots,i_n} g_{i_1,\dots,i_n} x^{i_1}_1 \cdots x^{i_n}_n \in K[\![x]\!],\\
	f(x,t) &= \sum_{i_1,\dots,i_n,j} f_{i_1,\dots,i_n,j} x^{i_1}_1 \cdots x^{i_n}_n t^j \in K[\![x,t]\!].
	\end{align*}
	Then the \emph{small diagonal} $\Delta(g)$ of $g(x)$ is the (univariate) formal power series given by: 
	\begin{align*}
	\Delta(g(x)) = \Delta(g(x))(t) \coloneqq \sum_{j \geq 0} g_{j,\dots,j} t^j \in K[\![t]\!].
	\end{align*}
	The \emph{big diagonal} $\mathcal{D}(f)$ of $f(x,t)$ is given by the (multivariate) formal power series:
	\begin{align*}
	\mathcal{D}(f(x,t)) = \mathcal{D}(f(x,t))(x) \coloneqq \sum_{i_1+ \cdots + i_n = j} f_{i_1,\dots,i_n,j} x^{i_1}_1 \cdots x^{i_n}_n \in K[\![x]\!].
	\end{align*}
\end{definition}
Clearly, for $n=2$ it holds that $\Delta(g(x_1,x_2)) = \mathcal{D}(g(x_1,t))$. We shall always refer to the big diagonal whenever we do not specify which diagonal we use.
{\setlength{\parindent}{0pt}
\begin{enumerate}[label=\textit{Example \arabic*:},wide]
	\item Let $x = x_1$ be one variable and $f(x,t) = 1/(1-x-t)$. Then we obtain
	\begin{align*}
		\mathcal{D}(f(x,t))(x) = \mathcal{D}\Bigl( \sum_{i,j \geq 0} \binom{i+j}{i} x^i t^j \Bigr)(x).
	\end{align*}
	Therefore it follows that $\mathcal{D}(f(x,t))(x) = \sum_{n\geq 0} \binom{2n}{n}x^n = (1-4x)^{-1/2}$. This function is an algebraic power series with minimal polynomial $P(x,t) = (1 - 4x)t^2 - 1.$
	\item Define the Hadamard product of two power series $f(x) = \sum_{\alpha \in \mathbb{N}^n} f_\alpha x^\alpha$ and $g(x) = \sum_{\alpha \in \mathbb{N}^n} g_\alpha x^\alpha$ to be the series $(f*g)(x) \coloneqq \sum_{\alpha \in \mathbb{N}} f_\alpha g_\alpha x^\alpha$. Now let again $x= x_1$ and $f(x,t) = \sum_{i,j \geq 0} c_{i,j} x^it^j$. Define $D \coloneqq \{ (i,j) \in \mathbb{N}^2 : i=j \}$ and its indicator function $\mathbbm{1}_D: \mathbb{N}^2 \to \{0,1\}$, then we obtain:
	\begin{align*}
		\Bigl( f(x,t)*\frac{1}{1-xt} \Bigr)(x,t) 
		& = \Bigl(\sum_{i,j\geq 0} c_{i,j} x^it^j * \sum_{i,j \geq 0} \mathbbm{1}_D x^it^j \Bigr)(x,t)  = \sum_{i,j\geq 0} c_{i,j} \mathbbm{1}_D x^it^j \\
		& = \sum_{n \geq 0} c_{n,n} (xt)^n  = \mathcal{D}(f(x,t))(xt),
	\end{align*}
	the diagonal of $f(x,t)$, with $xt$ substituted by $x$. This gives another viewpoint on the diagonal operator and was one historic reason for its definition.
	\item Consider the well-known power series
	\[
		f(t) = \sum_{n \geq 0} \sum_{k=0}^{n} \binom{n}{k}^2\binom{n+k}{k}^2 t^n \in \mathbb{Z}\ps{t} \subseteq \mathbb{C}\ps{t}.
	\]	
	This series appears in Apéry's proof of the irrationality of $\zeta(3)$. It is known to be transcendental and to satisfy a Picard-Fuchs differential equation, see \cite{AdAdDe16, AdAdDe13}. A lengthy but simple computation shows that $f(t)$ is the small diagonal of the rational power series
	\[
	\frac{1}{1-x_1}\cdot\frac{1}{(1-x_2)(1-x_3)(1-x_4)(1-x_5) - x_1x_2x_3} \in \mathbb{Z}\ps{x_1,x_2,x_3,x_4,x_5}.
	\]
In \cite{straub2014} a diagonal representation with four variables is shown:
\[f(t) = \Delta\left(\frac{1}{(1-x_1-x_2)(1-x_3-x_4) - x_1x_2x_3x_4}\right).\]	
Is is not known whether a similar rational expression using only 3 variables exists; in fact no power series is known for which 4 is provably the least number such that it is possible to write it as a small diagonal of some rational power series in that many variables \cite{BoLaSa17}. 
\end{enumerate}
}
There are many theorems in the literature connecting algebraic power series and diagonals of rational functions. For example, Pólya observed already in 1922 that the diagonal of any rational power series in two variables is necessarily algebraic \cite{polya1922}. Furstenberg's trick from 1967 implies that if $f \in K\ps{x}$ is algebraic and $x=x_1$ one variable, then there exists a rational power series $R(x,t)$ with $\mathcal{D}(R(x,t)) = f(x)$ \cite{Furstenberg67}. In the same paper, he proved that a small diagonal of any rational power series with coefficients in a field of positive characteristic is algebraic. In 1984 Deligne improved on the second result: the small diagonal of any \textit{algebraic} power series over a field of positive characteristic is algebraic \cite{Deligne1984}. Note that for fields of characteristic $0$ neither Deligne's nor Furstenberg's statements hold (cf. Example 3 above). Some elementary proofs of Deligne-Furstenberg's theorem have been found later by Harase \cite{Harase1988}, as well as by Sharif and Woodcock \cite{SharifWoodcock}. More recent and quantitative progress on this theorem is done by Adamczewski and Bell in \cite{AdAdDe13}. Denef and Lipshitz gave a simpler proof already in 1987 and generalized the first theorem of Furstenberg to several variables \cite{DeLi87}. This generalization stated below uses very abstract theory about the ring $K\aps{x}$ and we will present its proof in the last section using the discussed theorems of previous sections. A recent algorithmic confrontation to the viewpoint of algebraic power series as diagonals of rational functions is explained in \cite{BoDuSa17}. Finally, the reader can find Christol's survey about diagonals of rational functions in \cite{Christol}.
\begin{thm}[Denef \& Lipshitz] \label{DL}
	Let $f(x) \in K\langle x \rangle$ be an algebraic power series in $n$ variables over a field $K$ of characteristic zero.\footnote{In the original paper \cite{DeLi87} the statement is more general, allowing for excellent local integral domains instead of only fields of characteristic 0, but the ideas of the proof are the same in the special case we consider.} Then there exists a rational power series in $n+1$ variables $R(x,t) \in K(x,t) \cap K\ps{x,t}$ such that $f(x) = \mathcal{D}(R(x,t))$.\footnote{For completeness we mention that Denef and Lipshitz also proved another similar theorem in their paper. Using a slightly different notion of diagonal, they showed that an algebraic power series in $n$ variables is the diagonal of a rational function in $2n$ variables, see \cite[Theorem 6.2 (ii)]{DeLi87}.}
\end{thm}
While the statement of this theorem is completely elementary, its proof uses quite involved techniques of commutative algebra and algebraic geometry. The goal of this text is not only to motivate and explain them to a non-expert reader, but also to provide intuition for the main steps of the original proof. After the prefatory discussion about the ring of algebraic power series in Section \ref{sec:1}, we advance to Section \ref{sec:2} which is devoted to the notion of Henselization and its connection to $K\aps{x}$. In Section \ref{sec:3} we introduce étale ring maps and use this tool to prove an important fact about the Henselization of certain rings. Finally, in Section \ref{sec:4} we demonstrate the proof of Theorem \ref{DL} and present a new application of it, Theorem \ref{code}, in which we improve on the so-called Artin-Mazur lemma. We will recall all non-trivial definitions and try to be self-contained throughout the whole article.

Even though the proof of Theorem \ref{DL} is quite difficult, we can easily show it by the same trick as Furstenberg for a subclass of algebraic power series, which we call \emph{étale-algebraic}. As we will see, the difficulty is then reducing the general case to the étale-algebraic one. This is done by proving that any algebraic power series can be represented as a rational function in an étale-algebraic series.
\begin{definition}
	An algebraic power series $h(x) \in K\aps{x}$ with minimal polynomial $P(x,t) \in K[x,t]$ is called \emph{étale-algebraic} if $h(0) = 0$ and $\partial_t P(0,0) \neq 0$. 
\end{definition}
\noindent
Note that it immediately follows from the implicit function theorem that the minimal polynomial (as a function in $t$) of an étale-algebraic power series has a unique power series root vanishing at $0$, which must be the étale-algebraic series itself. Therefore, étale-algebraic power series are exactly those series which are encoded uniquely by their minimal polynomial. 

\textit{Example:} Take $x=x_1$; the algebraic power series $\tilde{h}(x) = x\sqrt{1-x}$ has minimal polynomial $\tilde{P}(x,t) = t^2+x^3-x^2$ which admits $\partial_t\tilde{P}(0,0) = 0$ and therefore $\tilde{h}(x)$ is not étale-algebraic. Clearly, $\tilde{P}(x,t) = (t - x\sqrt{1-x})(t+x\sqrt{1-x})$ has two power series solutions, both vanishing at $0$. However, consider $h(x) = \sqrt{1-x}-1$. Then still $h(0) = 0$ and for the new minimal polynomial we find $\partial_t P(0,0) = 2 \neq 0$, meaning that $h(x)$ is étale-algebraic.

\begin{lem} \label{diagonallemma}
Let $h \in K\aps{x}$ be étale-algebraic with minimal polynomial $P(x,t) \in K[x,t]$. Then the following rational function is a power series
\begin{align*}
	F \coloneqq t \frac{ \partial_tP(xt,t)}{P(xt,t)} \in K\ps{x,t}.
\end{align*}
Moreover, for any $i \in \mathbb{N}^n$ and $j \in \mathbb{N}$, 
the series
$x^i h(x)^j$ is the diagonal of $(xt)^i t^{j+1} F$.
In particular, Theorem~\ref{DL} holds if $f$ is étale-algebraic.
\end{lem}
\begin{proof}
	Since $h(x)$ is a root of $P(x,t)$, we can write $P(x,t) = (t-h(x))Q(x,t)$
	for some $Q(x,t) \in K[\![x]\!][t]$. Differentiating both sides with respect to $t$ gives
	\begin{align} \label{1}
	\partial_tP(x,t) = Q(x,t) + (t-h(x))\partial_t Q(x,t),
	\end{align} 
	hence, after dividing through by $P(x,t)/t$, we obtain
	\begin{align} \label{2}
	t \frac{\partial_tP(x,t)}{P(x,t)} = \frac{t}{t-h(x)} + t\frac{\partial_t Q(x,t)}{Q(x,t)}.
	\end{align}
	Equation $(\ref{1})$ implies that $Q(0,0) \neq 0$, therefore $Q(x,t) \in K[\![x,t]\!]^*$ and consequently the second summand in the equation above is a power series. Furthermore, since $h(0)=0$, we have
	\begin{align*}
	\frac{t}{t-h(xt)} = \frac{1}{1-t^{-1}h(xt)} \in K[\![x,t]\!].
	\end{align*}
	This concludes, using equation \eqref{2}, the proof of the first part of the Lemma. For the second part, consider
	\[
	    R(x,t) =(xt)^i t^{j+1} \frac{\partial_tP(xt,t)}{P(xt,t)}.
	\]
	A straightforward computation together with equation \eqref{2} yields 
\begin{align} \label{diagonalformula}
\begin{split}
	\mathcal{D} (R(x,t)) & = \mathcal{D}\Bigl( (xt)^i t^j \frac{1}{1-t^{-1}h(xt)} \Bigr) = \mathcal{D} \Bigl( (xt)^i \sum_{k \geq 0} t^{j-k} h(xt)^k \Bigr) \\
	& = \mathcal{D} \bigl( (xt)^i h(xt)^j \bigr) = x^i h(x)^j.
\end{split}
\end{align}
Finally, by setting $i=0$ and $j=1$ in~\eqref{diagonalformula}, we obtain the theorem of Denef and Lipshitz for étale-algebraic power series.
\end{proof}
\noindent
Now we dive into a more fundamental and basic study of our ring of interest $K\aps{x}$.

\subsection{Weierstrass Theorems}
The Weierstrass division theorem (WDT) and the preparation theorem (WPT) are fundamental classical results about the rings of formal and convergent power series. WDT is in some sense a form of the Euclidean division algorithm, but requires an extra property on the divisor. Often, applications and implications of the division algorithm for the polynomial ring can be translated via the Weierstrass division theorem to $K\ps{x}$. For example, it implies that the ring of formal power series is Noetherian, Henselian and a unique factorization domain. However, our main interest lies in the fact that both WDT and WPT hold true for algebraic power series and yield the same implications also for this ring. This fact was first proven by Lafon in 1965 \cite{lafon1} and is much less known. We shall reprove it following the ideas from \cite{LaTo70}.

\begin{definition} We introduce the following notation and object:
{\setlength{\parindent}{10pt}
\begin{enumerate}[label=\textit{(\arabic*)},wide]
	\item The \emph{order} of a non-zero formal power series $f = \sum_{\alpha \in \mathbb{N}^n}a_\alpha x^\alpha$, denoted by $\ord(f)$, is the smallest integer $d \geq 0$ such that $a_\alpha \neq 0$ for some $\alpha \in \mathbb{N}^n$ with $|\alpha| = d$. For $f = 0$ we say that $\ord(f) = + \infty$.
	\item A power series $g \in K\ps{x}$ is called \emph{$x_n$-regular of order $d$} if $g(0,\dots,0,x_n) = x_n^d f(x_n)$ for some power series $f(x_n) \in K\ps{x_n}$ with $f(0) \neq 0$.
	\item A polynomial $p \in K\ps{x'}[x_n]$ is called \emph{distinguished} if it is of the form $p = x_n^d + a_{d-1}(x')x_n^{d-1} + \cdots + a_0(x')$ for some power series $a_i(x') \in K\ps{x'}$ with $a_i(0) = 0$ for $i=0,\dots,d-1$.
\end{enumerate}
}
\end{definition}

\begin{thm}[WPT] \label{WPT}
	Let $g \in K\ps{x}$ be an $x_n$-regular power series of order $d$. Then there exist a unique $p \in K\ps{x'}[x_n]$ which is a distinguished polynomial in $x_n$ of degree $d$ and a unique unit $u \in K\ps{x}^*$, such that $g = up$. 
\end{thm}
\begin{thm}[WDT] \label{WDT}
	Let $g \in K\ps{x}$ be an $x_n$-regular power series of order $d$. For any $f \in K\ps{x}$ there exist uniquely a power series $q \in K\ps{x}$ and an $r \in K\ps{x'}[x_n]$ which is a distinguished polynomial in $x_n$ of degree less than $d$, such that $f = qg+r$.  
\end{thm}
\noindent
There are several different known ways to prove these classical theorems. The standard literature for their proofs and implications is the book ``The Basic Theory of Power Series'' by Ruiz \cite{Ruiz1993}, a more recent and very short proof can be found in \cite{Hauser17}. A very explicit approach is followed by Lang in \cite{Lang1984}. Now we state and prove Lafon's algebraic versions of these theorems.

\begin{thm}[Algebraic WPT] \label{algebraic_WPT}
	Let $g \in K\aps{x}$ be an $x_n$-regular algebraic power series of order $d$. Then there exist a unique $p \in K\aps{x'}[x_n]$ which is a distinguished polynomial in $x_n$ of degree $d$ with coefficients given by algebraic power series in $x'$ and a unique unit $u \in K\aps{x}^*$, such that $g = up$. 
\end{thm}
\begin{proof}
	First note that we may assume without loss of generality that $g$ is irreducible as a power series. Write $g=up$ with $u \in K\ps{x}^*$ and $p \in K\ps{x'}[x]$ a distinguished polynomial of degree $d$. It follows that $p$ has $d$ distinct roots in an algebraic closure of the quotient ring of the formal power series $\Omega = \overline{\Frac(K\ps{x'})}$, say $\alpha_1,\dots, \alpha_d$. Now let $G(x,t) = G_0(x) + G_1(x)t +\cdots +G_e(x) t^e \in K[x,t]$, $G_0 \neq 0$ be a minimal polynomial of $g$, i.e. we have
	\begin{align*}
	0 = G(x,g(x)) = G_0(x) + G_1(x)g(x) +\cdots +G_e(x) g(x)^e.
	\end{align*}
	For every $i = 1,\dots,d$ we can replace $x$ by $(x',\alpha_i)$ and, using $g(x',\alpha_i) = 0 $, we obtain for every of those $i$'s
	\begin{align*}
	0 = G(x',\alpha_i,g(x',\alpha_i)) = G_0(x', \alpha_i).
	\end{align*}
	As $0 \not \equiv G_0(x',x_n) \in K[x',x_n] \subseteq K(x',x_n)$ and annihilates $\alpha_i$, we get that $\alpha_i$ is algebraic over $K(x')$. It follows that $x_n - \alpha_i$ is algebraic over $K(x',x_n) = K(x)$. Therefore, $p = \prod_{j=1}^{d} (x_n - \alpha_j)$ is an algebraic power series and the same holds for $u = g/p$. Uniqueness is guaranteed by uniqueness of Weierstrass formal preparation (Theorem \ref{WPT}).
\end{proof}
\begin{thm}[Algebraic WDT]  \label{algebraic_WDT}
	Let $g \in K\aps{x}$ be an $x_n$-regular algebraic power series of order $d$. For any $f \in K\aps{x}$ there exist uniquely an algebraic power series $q \in K\aps{x}$ and an $r \in K\aps{x'}[x_n]$ which is a distinguished polynomial in $x_n$ of degree less than $d$ with coefficients given by algebraic power series in $x'$, such that $f = qg+r$.  
\end{thm}
\begin{proof}
	Again, we may assume that $g$ is irreducible and by the algebraic Weierstrass preparation theorem, we may also assume without loss of generality that $g \in K\aps{x'}[x_n]$ is a distinguished polynomial of degree $d$. We can divide formally:
	\begin{align} \label{fqgr}
	f = qg+r = qg+\sum_{j = 0}^{d-1} b_j(x') x_n^j,
	\end{align}
	for $q \in K\ps{x}$ and $b_0(x'), b_1(x'), \dots , b_{d-1}(x') \in K\ps{x'}$ formal power series. Because $g$ is assumed to be a distinguished polynomial, we may write $g = \sum_{j=0}^d c_j(x')x_n^j$ for some $c_0(x'), \dots , c_{d}(x') \in K\aps{x'}$ algebraic power series. We get that $g$ has $d$ distinct roots in $\Omega = \overline{\Frac(K\ps{x'})}$, say $\alpha_1,\dots, \alpha_d$. From $(\ref{fqgr})$, by replacing $x_n$ with $\alpha_i$, we get for every $i = 1, \dots , d$ that $f(x',\alpha_i) = \sum_{j=0}^{d-1} b_j(x')\alpha_i^j$. This can be rephrased in terms of a matrix multiplication: 
	\begin{align*}
	\begin{pmatrix}
	f(x', \alpha_1)\\
	\vdots \\
	f(x', \alpha_d)
	\end{pmatrix} 
	= 	
	\begin{pmatrix}
	1 & \alpha_1 & \cdots & \alpha_1^{d-1} \\
	1 & \alpha_2 & \cdots & \alpha_2^{d-1} \\
	\vdots & \vdots & \ddots &  \vdots \\
	1 &  \alpha_d &  \cdots & \alpha_d^{d-1}
	\end{pmatrix} 	
	\begin{pmatrix}
	b_0(x')\\
	\vdots \\
	b_{d-1}(x')
	\end{pmatrix} .
	\end{align*}
	The matrix above is the Vandermonde matrix and it is invertible since the $\alpha_i$'s are pairwise different. Therefore, each $b_i(x')$ is uniquely given by some rational expression in the $f(x', \alpha_j)$'s and $\alpha_k$'s for $j,k \in \{ 1,\dots, d\}$. On the other hand, by the same argument as in the proof of Theorem \ref{algebraic_WPT}, we obtain that each $\alpha_i$ is algebraic over $K(x')$. Moreover, by assumption $f(x',x_n)$ is algebraic over $K(x',x_n)$. It follows that both field extensions $K(x') \subseteq K(x',\alpha_i)$ and $K(x',\alpha_i) \subseteq K(x',\alpha_i,f(x', \alpha_i))$ are finite. Hence, $K(x') \subseteq K(x', f(x', \alpha_i))$ is finite and consequently $f(x',\alpha_i)$ is algebraic over $K(x')$. As this holds for every $i = 1,\dots,d$, it follows that also any rational expression in the $f(x',\alpha_j)$'s and $\alpha_k$'s for $j,k \in \{1,\dots, d\}$ is algebraic over $K(x')$. Recall that each $b_i(x')$ is such an expression, therefore each $b_i(x')$ is algebraic over $K(x')$ and hence an algebraic power series. It follows that $r = \sum_{j = 0}^{d-1} b_j(x') x_n^j$ is an algebraic power series and finally the same holds for $q = (f-r)/g$. Uniqueness is clear by the formal WDT (Theorem \ref{WDT}).
\end{proof}
As already mentioned, this theorem has many implications. For example, the proofs from \cite{Ruiz1993} and \cite{Lang1984} that $K\ps{x}$ is Noetherian and a UFD can be directly carried out for $K\aps{x}$. Another corollary is the following:
\begin{thm}[Hensel's Lemma] \label{hensel_aps0}
	Let $f \in K\aps{x}[t]$ be a monic polynomial in $t$ over $K\aps{x}$. Assume $\alpha \in K$ is a root of multiplicity $d$ of the polynomial $f(0,t) \in K[t]$. Then there exist unique monic polynomials $p, u \in K\aps{x}[t]$ with $u(0,\alpha) \neq 0$, $p$ of degree $d$ in $t$, $p(0,t) = (t-\alpha)^d$ and $f = up$.
\end{thm}
\begin{proof}
	After the change of the variable $t$ to $t+\alpha$, we may assume that $\alpha=0$. Since $\alpha$ is a $d$-multiple root of $f(0,t)$, it follows that $f(x,t)$ is $t$-regular of order $d$. By the algebraic WPT in $n+1$ variables we may write uniquely $f = u p$ where $p \in K\aps{x}[t]$ is a distinguished polynomial in $t$ of degree $d$ and $u \in K\aps{x,t}^*$ a unit, hence $u(0,\alpha) = u(0,0) \neq 0$. Moreover, since $p(x,t)$ is distinguished of degree $d$ it follows by definition that $p(0,t) = t^d = (t-\alpha)^d$. $u(x,t) \in K\aps{x}[t]$ because of polynomial division in this ring. Uniqueness follows from the uniqueness of the algebraic Weierstrass division theorem.
\end{proof}

The statement above ensures that a root $\alpha \in K$ of $f(0,t)$ gives rise to a factorization $f = up$. This factorization is called the \emph{lifting} of $\alpha$. Lifting all roots separately, one by one, it is possible to prove that coprime factorizations can be lifted as well.
\begin{thm}[Hensel's Lemma] \label{hensel_aps1}
	Let $f \in K\aps{x}[t]$ be a monic polynomial in $t$ over $K\aps{x}$. Assume $f(0,t) = \bar{p}(t)\bar{q}(t)$ factors into two monic coprime polynomials. Then there exist two unique monic polynomials $p, q \in K\aps{x}[t]$ with $p(0,t) = \bar{p}(t)$, $q(0,t) = \bar{q}(t)$ and $f = qp$. 
\end{thm}
\noindent
We omit the detailed proof here, because the equivalence on these two versions of Hensel's lemma will be justified in Section \ref{sec:3} in a more general setting. Instead, we explain the importance of these theorems following the motivation in \cite[pp. 185]{eisenbud1995commutative}.

\subsection{The Importance of Hensel's Lemma}

\begin{minipage}{.65\linewidth}
Consider the nodal plane cubic curve over a field $K$ (of characteristic $0$ as always or positive but not equal to $2$) given by the equation $t^2-x^2(1+x) = 0$ for $x=x_1$. The associated affine coordinate ring is $S = K[x,t]/(t^2-x^2(1+x))$. Of course, the curve is irreducible and $S$ is a domain. When looking at the picture over $\mathbb{R}$ (Figure \ref{node}), one may think that localizing $S$ at the maximal ideal $\idm = (\bar{x},\bar{t})$ will make the ring have zero divisors, however this is not the case: every Zariski neighborhood of $0$ of the node is irreducible. The reason is that over the complex numbers a neighborhood of the omitted origin is a punctured disc and therefore the curve remains irreducible. We would still like to factor $t^2-x^2(1+x)$ somehow, in order to study the easier rings into which $S$ will decompose. Examining a ``really small neighborhood'' of the node, we would expect the curve to become reducible there: for example over the ring of formal power series the expression $t^2-x^2(1+x)$ is in fact reducible. This comes from the fact that $1+x$ has a square root in $K\ps{x}$ and we may therefore write 
\end{minipage}\hfill
\begin{minipage}{.3\linewidth}
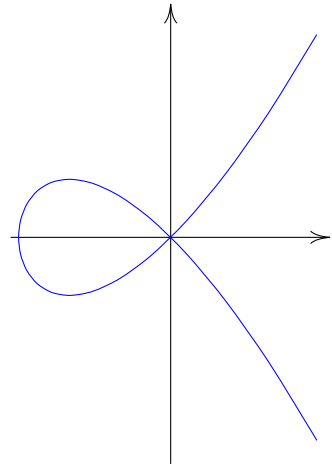
\begin{figure}[H] 
\centering
\begin{tikzpicture}[line cap=round,line join=round,x=1.0cm,y=1.0cm]

\draw[color=black,-{>[scale=2.5,
	length=3,
	width=2]}] (-2.1,0.) -- (2.1,0.);

\draw[color=black,-{>[scale=2.5,
	length=3,
	width=2]}] (0.,-3) -- (0.,3.1);

\draw[scale=2,domain=-1.4:1.4,smooth,variable=\x,blue] plot ({\x*\x-1},{\x*(\x*\x-1)});
\end{tikzpicture}
\caption{The node: $\mathbb{R}[x,t]/(t^2-x^2(1+x))$.}	\label{node}
\end{figure}
\end{minipage}
$t^2-x^2(1+x) = (t-x\sqrt{1+x})(t+x\sqrt{1+x})$. One can argue the reason why it is immediately clear that $1+x$ is a square over $K\ps{x}$ is that this ring satisfies Hensel's lemma! More precisely, take the polynomial $f(x,t) = t^2-(1+x)$, then $f(0,t) = (t-1)(t+1) = \bar{p}(t)\bar{q}(t)$ and these polynomials are coprime. Therefore, by Hensel's lemma, this factorization must admit a lifting and therefore $\sqrt{1+x} \in K\ps{x}$. This is also the reason why we explicitly do not allow the characteristic of $K$ to be $2$: in this case $\bar{p}(t)$ and $\bar{q}(t)$ would not be relatively prime and the lifting would not be guaranteed, in fact it would not exist. 
However, we also see that in order to make the node reducible, we do not have to go from $K[x,t]$ all the way up to the polynomial ring over the completion $K\ps{x}[t]$: it suffices to take \textit{any} Henselian ring extension of $S$ or of $K[x]_{(x)}$. This is exactly the idea and motivation for defining the Henselization.

\section{Algebraic Power Series and Henselization} \label{sec:2}
The main goal of this section is to stress the connection between the algebraic closure in the completion of a ring and the property of being Henselian, that is to satisfy Hensel's lemma. We will be able to prove that, under certain conditions, any Henselian ring is algebraically closed in its completion, that is, if $a \in \widehat{R}$ is algebraic over a Henselian $R$ then it must already hold that $a \in R$. This will allow us to look at the ring of algebraic power series from a different viewpoint. The conditions may appear technical at first sight, however they have the purpose of excluding pathologies while still allowing for a large class of rings. A considerably different approach to this theory can be found in the book \cite{KuPfRo75} by Kurke, Pfister and Roczen.

We will work with \emph{local rings}, i.e with those rings $R$, which have exactly one maximal ideal. Usually, we will denote this maximal ideal by $\idm$ and let $K \coloneqq R/\idm$ be the residue field with respect to $\idm$. Sometimes, we write triples $(R, \idm, K)$ when talking about local rings, combining these three objects. It is immediate to see that $R$ has only one maximal ideal $\idm$ if and only if $R^* = R\setminus\idm$. Recall that given two local rings $(R,\idm,K), (S,\idn,L)$, a homomorphism $\phi: R \to S$ is called \emph{local} if $\phi(\idm) \subseteq \idn$ holds and this condition is equivalent to $\phi^{-1}(\idn) = \idm$. Note that both, the rings of formal and algebraic power series, are local with maximal ideal $\idm = (x) = (x_1,\dots,x_n)$. It is also obvious that the residue field $K\ps{x}/(x) \cong K\aps{x}/(x)$ is (isomorphic to) $K$. Recall that the \textit{completion} of a local ring $(R,\idm,K)$ is defined as the inverse limit $\invlim R/\idm^i$. A local ring $(R,\idm,K)$ is called \textit{complete} if the canonical map to its completion $R \to \widehat{R}$ is an isomorphism. Note that $\cap_i \idm^i$ goes to zero under this mapping, therefore completeness implies $\cap_i \idm^i = 0$. Krull's intersection theorem justifies that the completion of a \textit{Noetherian} ring (i.e. a ring with only finitely generated ideals) is complete, however this is false in general. Finally, recall that the completion $\widehat{R}$ of $R$ satisfies the universal property that for any local $(R,\idm) \to (S,\idn)$ with $(S,\idn)$ complete there exists a unique factorization $R \to \widehat{R} \to S$.

\subsection{Henselian Rings}
Given a $p \in R[t]$, we will denote by $\bar{p} \in K[t]$ the reduction of $p$ mod $\idm R[t]$, given by reducing all coefficients of $p$ mod $\idm$.
\begin{definition} \label{henselian_defi}
	A local ring $(R,\idm,K)$ is called \emph{Henselian} if the following property holds:\\
	Let $f(t) \in R[t]$ be a monic polynomial. Assume that $\bar{f}(t) = p_0(t)q_0(t)$ holds for two monic coprime polynomials ${p_0}(t), q_0(t) \in K[t]$. Then there exist two unique monic polynomials $p(t), q(t) \in R[t]$ satisfying $\bar{p}(t) = p_0(t),\bar{q}(t)=q_0(t)$, $\deg p(t) = \deg p_0(t)$, $\deg q(t) = \deg q_0(t)$ and $f(t) = p(t)q(t)$.
\end{definition}
\noindent
The notion of Henselian rings (or Hensel rings) was first introduced by Azumaya in \cite{azumaya1951}. The property above is usually referred to as ``Hensel's lemma'' even though it is used as a definition. One often reads in the literature ``A ring is called Henselian, if Hensel's lemma holds [in this ring]''. To avoid confusion, we will call the statement above Hensel's property. The actual ``lemma'' of Hensel is the following classical theorem (see for example \cite[Chapter 7]{eisenbud1995commutative}):
\begin{thm} [Hensel's lemma for complete rings]
	Let $(R,\idm,K)$ be a complete local ring. Then $R$ is Henselian. 
\end{thm}
\noindent Standard references for Henselian rings are \cite{nagata1975local, raynaud1970anneaux, Grothendieck67}.
For the purpose of this work, a very significant fact is that Theorem $\ref{hensel_aps1}$ implies the following:
\begin{thm} \label{algebraic_henselian}
	The ring of algebraic power series $K\aps{x}$ is Henselian. 
\end{thm}
Another key fact about Henselian rings is the following lemma. Its statement and proof appear in \cite{nagata1975local} and give an introductory flavor to this section. Recall that given an extension of rings $R \subseteq S$, an element $s \in S$ is called \textit{integral over $R$} if there exists a monic polynomial $P(t) \in R[t]$ with $P(s) = 0$. The extension is said to be \textit{integral} if this holds for any $s \in S$.
\begin{lem} \label{30.5}
	Let $R$ be a Henselian integral domain and $R'$ an integral extension of $R$. Then $R'$ is a local ring.
\end{lem}
\begin{proof}
	Let $\idm$ be the maximal ideal of $R$ and assume that $R'$ has two maximal ideals $\idm_1' \neq \idm_2'$. Take some $a \in \idm_1'$ which is not in $\idm_2'$. We have an irreducible monic polynomial $f(t) = t^n + c_{n-1}t^{n-1} + \cdots + c_0 \in R[t]$ which has $a$ as root. Now, as $a \in \idm_1'$, we must have $c_0 \in \idm_1' \cap R \subseteq \idm$. We also have that $a^n \not \in \idm R[a]$, because $a \not \in \idm_2'$, hence there must be a $c_i$ which is not in $\idm$. Take $j \in \mathbb{N}$ such that $c_j \not \in \idm$ but $c_{j-s} \in \idm$ for $0<s\leq j$. Clearly $1\leq j \leq n-1$ and we have
	\[
	f(t) \equiv (t^j + c_{n-1}t^{j-1} + \cdots + c_{n-j})t^{n-j} \mod \idm R[t].
	\]
	But this means that the image of $f$ is reducible mod $\idm R[t]$ and using that $R$ is Henselian we obtain that $f$ must be reducible in $R[t]$. This is a contradiction and the assertion is proved. 
\end{proof}
To state the main theorem of this section we will need some conditions on our local ring $R$. Therefore we define the necessary terms:
\begin{definition} Let $R$ be a ring.
{\setlength{\parindent}{10pt}
\begin{enumerate}[label=\textit{(\arabic*)},wide]
	\item $R$ is called \emph{analytically irreducible} if it is local and its completion $\widehat{R}$ is a domain. $R$ is called \emph{analytically normal} if it is local and $\widehat{R}$ is normal\footnote{Recall that a local ring is called normal if is integrally closed in its quotient field.}.
	\item Assume $R$ be an integral domain with quotient field $L$. $R$ is called \emph{Japanese}\footnote{According to \cite{stacks-project} this name was first used by Grothendieck in order to contribute to Nakayama, Takagi, Nagata and many others.} if it satisfies the so-called finiteness condition for integral extensions. This means, for every finite extension $L'$ of the quotient field $L$, the integral closure of $R$ in $L'$ is a finitely generated $R$-module.
	\item $R$ is called a \emph{Nagata ring}\footnote{In the book ``Local Rings'' Nagata calls these rings pseudo-geometric.} if $R$ is Noetherian and for every prime ideal $\idp \subseteq R$, the ring $R/\idp$ is Japanese. 
\end{enumerate}
}
\end{definition}
\begin{lem} \label{36.1}
	If $R$ is a Nagata ring, then every ring which is a finite module over $R$ or a ring of quotients of $R$ is also Nagata. Any localization of $R$ is also Nagata. Any finitely generated $R$-algebra is Nagata. 
\end{lem}
\noindent
We see that the category of Nagata rings is reasonably large and closed under many operations. Good references for the proofs of the lemma above are \cite[Section 36]{nagata1975local} and \cite[Section 31]{matsumura1980commutative}. Of course, most facts can also be found in \cite[\href{https://stacks.math.columbia.edu/tag/032U}{Section 032E}]{stacks-project}.\\

Finally, for our purposes we need the following version of the Zariski Main Theorem, which is Theorem 37.8 in \cite{nagata1975local}. Recall that $S$ is said to be of \textit{finite type} over $R$ if $S$ is isomorphic to a quotient of $R[x_1,\dots,x_n]$ as an $R$-algebra.
\begin{lem} \label{37.8}
	Let $R$ be an analytically normal ring. If a normal and local Nagata ring $S$ is of finite type over $R$, then $S$ analytically irreducible.
\end{lem}
We are ready to prove one central theorem of this section, connecting Henselian rings with the property of being algebraically closed in the completion, and which  will be a crucial ingredient in the proof of Theorem \ref{hensel=algebraic}. This theorem explains why the study of algebraic power series essentially comes down to studying Henselian rings. Its statement can be found in \cite[(44.1)]{nagata1975local}, however the proof given there is very concise and in some places unclear. Therefore we shall reprove the theorem here.
\begin{thm} \label{closed_in_comp}
	Let $R$ be a Henselian, analytically normal Nagata ring. Then $R$ is algebraically closed in its completion, i.e. if $a \in \widehat{R}$ algebraic over $R$, then $a \in R$. 
\end{thm}
\begin{proof}
	Let $a$ be an element of $\widehat{R}$ which is algebraic over $R$. Then we can find $b \neq 0$ in $R$ such that $c \coloneqq ab$ is integral over $R$. We claim that $R[c] = R$. Assume otherwise and let $f \in R[t]$ be the minimal polynomial of $c$. We wish to use the lemma above on $R[c]$, but we lack the assumption of normality. So we define $R'$ to be the integral closure of $R[c]$ in $L[c]$, where $L = \Frac(R)$, so $R'$ is normal by definition. By an easy observation it follows that $R'$ is also the integral closure of $R$ in $L[c]$. Since $R$ is analytically normal, it is a domain and therefore the ideal $(0)$ is prime. By the definition of a Nagata ring, it follows that the integral closure of $R = R/(0)$ in any finite extension of $L$ is a finitely generated $R$-module. Since $c$ is integral and in particular algebraic, it follows that $L[c]$ is a finite extension of $L$ and therefore $R'$ is a finitely generated $R$-module. Furthermore, $R'$, being finitely generated over a Nagata ring, is still Nagata by Proposition \ref{36.1} and also $R'$ is indeed local by Lemma \ref{30.5}\footnote{This is where we use the Henselian assumption.}, hence we may apply Lemma \ref{37.8} to get that $R'$ is analytically irreducible. However, we have that $\widehat{R'} = R' \otimes_R \widehat{R}$ and this must be a domain. Now look at the completion of $R[c]$, which is given by $R[c] \otimes_R \widehat{R}$. Now, $R \to \widehat{R}$ is flat, meaning that we also have the inclusion $R[c] \otimes_R \widehat{R} \subseteq R' \otimes_R \widehat{R}$ and hence $\widehat{R[c]}$ must be a domain as well. On the other hand, we have $\widehat{R[c]} = R[c] \otimes_R \widehat{R} = \widehat{R}[t]/(f)$, identifying $f$ with its image in $\widehat{R}[t]$. However $c \in \widehat{R}$ is a root of $f$, hence $\widehat{R}[t]/(f)$ cannot be a domain: a contradiction. So $c \in R$ and hence $a \in \Frac(R)$. Because $R = \Frac(R) \cap \widehat{R}$ and since $a$ is in both rings, we get that $a \in R$ as wanted.
\end{proof}
\subsection{Henselian Characterization of Algebraic Power Series}
We saw that Henselian rings are closely connected to algebraic closures in the completion. In particular, at this point, one may conjecture that for some, not necessarily Henselian, ring $R$, if we can define the ``smallest'' Henselian extension of $R$, it will be exactly the algebraic closure of $R$ in $\widehat{R}$. Since algebraic power series are by definition the algebraic closure of $K[x]_{(x)}$ in its completion, this approach will also give a different viewpoint on our main ring of interest. Note that obviously any field is Nagata, therefore by Lemma \ref{36.1} it follows that $K[x]$ is also a Nagata ring. Then $K[x]_{(x)}$ is again Nagata, since it is a localization. 
\begin{definition}
	Let $(R,\idm,K)$ be a local ring. We say a Henselian ring $\hen{R}$ together with a local homomorphism $i: R \to \hen{R}$ is the \emph{Henselization} of $R$, if any local homomorphism from $R$ to a Henselian ring factors uniquely through $i$.
\end{definition}
In other words, the Henselian ring $\hen{R}$ together with $i: R \to \hen{R}$ is the Henselization of $R$, if for any Henselian ring $H$ and local $\psi: R \to H$ there exists a unique local $\phi$ such that the following diagram commutes:
\[
\begin{tikzcd}
	R \arrow[r, "i"] \arrow[d, "\psi", swap] & \hen{R} \arrow[dl, dashed, "\phi"] \\
	H  
\end{tikzcd}
\]
This notion was first introduced by Nagata in the article ``On the theory of Henselian rings'', which became the first of a trilogy \cite{Nagata1953, Nagata1954, Nagata1959}. Since then, the Henselization of a ring became a very well studied object; we shall only explain those facts which are of importance for our purpose. 

Note that from the definition it follows that if $\hen{R}$ exists, then it must be unique up to isomorphism. For Noetherian rings one has the inclusion $R \hookrightarrow \widehat{R}$; this together with Hensel's lemma immediately implies that $i$ must be injective as well in this case. Moreover, the following fact follows also easily from the universal property: Assume the existence of a Henselian ring $R'$ such that $R \subseteq R' \subseteq \hen{R}$, then $R' = \hen{R}$. In this sense we can view the Henselization as the ``smallest'' Henselian ring extension of $R$. Note that it is not obvious that $\hen{R}$ exists for any local $R$, however this is true and we will prove this in the next section (Section \ref{3.3}). 

The following theorem allows a purely ring-theoretic viewpoint on $K\aps{x}$.
\begin{thm} \label{hensel=algebraic}
	Let $R = K[x]_{(x)}$ be the localization of $K[x]$ at the maximal ideal $(x)$. Assume that the Henselization of $R$ exists\footnote{In the next section (Section \ref{3.3}) we will prove that the Henselization of a local ring always exists.}. Then it is isomorphic to the ring of algebraic power series: $\hen{R} \cong K \langle x \rangle$.
\end{thm}
\noindent
For the proof we need two lemmas: we provide a proof for the first, and a reference for the second. 
\begin{lem} \label{completions}
	Let $\hen{R}$ be the Henselization of a Noetherian local ring $R$. Then $\widehat{\hen{R}} = \widehat{R}$.
\end{lem} 

\begin{proof}
	By Hensel's lemma $\widehat{R}$ is Henselian, thus there exists a unique factorization $R \to \hen{R} \to \widehat{R}$. Let $S$ be a complete local ring with maximal ideal $\idn$ and assume $\hen{R} \to S$ is a local map. Precomposing with $i$ gives $R \to \hen{R} \to S$. By the universal property of the completion and then by the factorization of $R \to \widehat{R}$ we also find $R \to \hen{R} \to \widehat{R} \to S$. Now, since $S$ being complete is also Henselian, the uniqueness in the universal property of $\hen{R}$ implies that these factorizations are equal. Hence, for every $R \to S$ with $S$ complete we find $\hen{R} \to \widehat{R} \to S$. The universal property of the completion forces $\widehat{\hen{R}} = \widehat{R}$. 
\end{proof}

\begin{lem} \label{44.2}
	Let $R$ be a local Nagata ring. Then its Henselization $\hen{R}$ is also Nagata. Moreover, if $R$ is also analytically normal then so is $\hen{R}$.
\end{lem}
\noindent 
For the proof see \cite[(44.2,44.3)]{nagata1975local}.

\begin{proof}[Proof of Theorem \ref{hensel=algebraic}]
    By Lemma \ref{44.2} it follows that $\hen{K[x]_{(x)}}$ is both analytically normal and Nagata. Because the ring of algebraic power series is Henselian (Theorem \ref{algebraic_henselian}) and the Henselization is the smallest Henselian ring extension of a given ring, it suffices to show $K\aps{x} \subseteq \hen{K[x]_{(x)}}$. Let $f \in K\aps{x} \subseteq K\ps{x} = \widehat{K[x]_{(x)}}$. Then obviously $f$ is algebraic over $\hen{K[x]_{(x)}}$ and by Lemma \ref{completions} we must have $f \in \widehat{\hen{K[x]_{(x)}}}$. We can apply Theorem \ref{closed_in_comp} to see that $f \in \hen{K[x]_{(x)}}$.
\end{proof}

\section{Étale Ring Maps and Henselization} \label{sec:3}
\subsection{Motivation for Étale Ring Maps}
Before giving the rigorous definition of an étale map $R \to S$ between two rings $R,S$, we will try to explain the motivation behind it. Milne writes in his lecture notes \cite{milneLEC}:
\begin{quote}
	\textit{``An étale morphism is the analogue in algebraic geometry of a local isomorphism of manifolds in differential geometry, a covering of Riemann surfaces with no branch points in complex analysis, and an unramified extension in algebraic number theory.''}
\end{quote}
Of course, the importance of these objects makes it clear that one needs a definition in the setting of algebraic geometry and that this definition might be involved. There are many equivalent ways to define this analogue and we will try to motivate the one that is mostly geometric and closest to a universal property.

Consider the case of two affine algebraic varieties $X = V(f_1,\dots,f_r) \subseteq K^n, Y = V(g_1,\dots,g_s) \subseteq K^m$ and a morphism  $f_\phi: X \to Y$ coming from $\phi: R \to S$, where 
\begin{align*}
R \coloneqq K[Y] & = K[y_1,\dots, y_m]/(g_1,\dots,g_s) \text{ and }\\
S \coloneqq K[X] & = K[x_1,\dots,x_n]/(f_1,\dots,f_r)
\end{align*}
are the corresponding coordinate rings. Recall that a local diffeomorphism is characterized by its bijective differential. We want to achieve an analogous property for $f_\phi$ by putting only algebraic conditions on $\phi$.

By definition, $f_\phi$ maps any $K$-point $a \coloneqq (a_1,\dots,a_n) \in X$ to a $K$-point $b \coloneqq (b_1,\dots,b_m) \in Y$. To formulate this in an algebraic way, we can require the following diagram to commute:
\[\begin{tikzcd}
\llap{$S = $ } K[X] \arrow[r] & K \\
\llap{$R = $ } K[Y] \arrow[ur] \arrow[u, "\phi"] & 
\end{tikzcd}
\]
To see that this algebraic formulation indeed corresponds to the geometric viewpoint of sending $a \in X$ to some $b \in Y$, note that the map $K[X] \to K$ defines a $K$-point of $X$, since it maps each $x_i$ to $a_i$ for some $a \coloneqq (a_1,\dots,a_n) \in K^n$ with the condition that each $f_j(a_1,\dots,a_n) = 0$, $1\leq j \leq r$, hence, by definition, $a \in X$. Similarly, $K[Y] \to K$ is a $K$-point, say $b = (b_1,\dots,b_m) \in Y$, because $g_j(b_1,\dots,b_m) = 0$ for $j=1,\dots,s$. The commutativity of the diagram means that sending $(y_1,\dots,y_m) \mapsto (b_1,\dots,b_m)$ by the diagonal map is the same as sending $(y_1\dots,y_m) \mapsto (\phi_1(x_1, \dots, x_n),\dots, \phi_m(x_1, \dots, x_n)) \mapsto (\phi_1(a),\dots, \phi_m(a))$: $K$-points are sent to $K$-points.

Now we want to describe the behavior of $f_\phi$ on tangent vectors. We can formulate this in an algebraic way, by requiring the commutativity of the following diagram, adding the ring $K[\varepsilon]/(\varepsilon^2)$ to the above:
\[\begin{tikzcd}
\llap{$S = $ } K[X] \arrow[r] & K \\
\llap{$R = $ } K[Y] \arrow[r] \arrow[u, "\phi"] & K[\varepsilon]/(\varepsilon^2) \arrow[u]
\end{tikzcd}
\]
Since 
\begin{align*}
	K[Y] & \xrightarrow{\hspace*{1cm}} K[\varepsilon]/(\varepsilon^2) \xrightarrow{\hspace*{1.2cm}} K\\
	(y_1,\dots,y_m) & \mapsto (b_1+ \varepsilon c_1, \dots, b_m+\varepsilon c_m ) \mapsto (b_1,\dots,b_m),
\end{align*}
we see that this intermediate ring does not destroy the considerations above. Moreover, we claim that the map $K[Y] \to K[\varepsilon]/(\varepsilon^2)$ corresponds to a tangent vector of $Y$: say, we have 
\begin{align*}
	K[Y] = K[y_1,\dots,y_m]/(g_1,\dots,g_s) &\to K[\varepsilon]/(\varepsilon^2)\\
	y_i & \mapsto  b_i + \varepsilon c_i, \hspace{0.5cm} 1\leq i \leq m,
\end{align*}
for some $b \in K^m$ and $c \coloneqq (c_1,\dots, c_m) \in K^m$. Then it must hold that $g_j(b_1 + \varepsilon c_1, \dots, b_m + \varepsilon c_m) = 0$ for $1 \leq j \leq s$. Using Taylor expansion and the fact that $\varepsilon^2 =0$ in $K[\varepsilon]/(\varepsilon^2)$, we obtain:
\[
	0 = g_j(b_1 + \varepsilon c_1,\dots,b_m+\varepsilon c_m) = g_j(b) + \sum_{i=1}^m \frac{\partial g_j}{\partial y_i}(b) c_i \varepsilon.
\]
Comparison of the coefficients in $\varepsilon$ gives that $g_j(b_1,\dots ,b_n) = 0$ for each $j$, i.e. $b$ is a $K$-point of $Y$ (what we already knew), and that 
\[
	\sum_{i=1}^m \frac{\partial g_j}{\partial y_i}(b) c_i = 0, \hspace{0.5cm} 1\leq j \leq s.
\]
This is of course equivalent to $c \cdot \nabla g_j(b) = 0$, i.e. $c$ is a tangent vector of $Y$ at $b$ and we may say $c \in T_bY$, the tangent space of $Y$ at $b$.

Up to now, we have reformulated the property of $f_\phi$ to map $K$-points to $K$-points and added the potential of considering tangent vectors in terms of a commutative diagram. We can now add the final requirement to $\phi$, making it the analogue of a local diffeomorphism: we want its ``differential'' $T_aX \to T_{f_\phi(a)}Y = T_bY$ to be bijective. Surprisingly, this condition is very easy to add in our commutative diagram formalism: we require additionally the existence and uniqueness of the diagonal arrow, preserving commutativity:
\[\begin{tikzcd}
\llap{$S = $ } K[X] \arrow[r] \arrow[rd] & K \\
\llap{$R = $ } K[Y] \arrow[r] \arrow[u, "\phi"] & K[\varepsilon]/(\varepsilon^2) \arrow[u]
\end{tikzcd}
\]
By the same argument as above, we can easily convince ourselves that this diagonal map writes $K[X] \to K[\varepsilon]/(\varepsilon^2): x_i \mapsto a_i + \varepsilon d_i$ for $i = 1,\dots, n$ and some $d \coloneqq (d_1,\dots,d_n)$ which corresponds to a tangent vector of $X$. The commutativity of the upper-right triangle just means that this vector is in the tangent space $T_aX$. Finally, consider the commutativity of the lower triangle. On the one hand, we can map by the horizontal homomorphism $y_j \mapsto b_j + \varepsilon c_j$, $1 \leq j\leq m$ as we already saw. On the other hand, going the other path, we have again by Taylor's expansion for $j=1,\dots,m$:
\[
y_j \mapsto \phi_j(x_1,\dots,x_n) \mapsto \phi_j(a_1+\varepsilon d_1,\dots,a_n+\varepsilon d_n) = \phi_j(a) + \sum_{i=1}^n \frac{\partial \phi_j}{\partial x_i}(a) d_i \varepsilon.
\]
Since the lower-left triangle commutes, we have by comparison of the coefficient of $\varepsilon$ that 
\[
	c_j = \sum_{i=1}^n \frac{\partial \phi_j}{\partial x_i}(a) d_i, \hspace{0.5cm} 1\leq j \leq m. 
\]
Putting these $m$ equations together, we define the Jacobian matrix 
\[
J_\phi(a) \coloneqq \left( \frac{\partial \phi_j}{\partial x_i}(a)  \right)_{\substack{1\leq i \leq n\\1\leq j \leq m}}.
\]
Then, the equation above is, of course, equivalent to $J_\phi(a) d = c$. 

Hence, the existence of the diagonal arrow makes sure that for any tangent vector at $b \in Y$, we have at least one tangent vector at $a \in X$ mapping to it, in other words it ensures the surjectivity of $J_\phi(a)$. Analogously, the uniqueness of the diagonal map translates into injectivity of the differential. Equipped with this good understanding of what it means to define the algebraic analogue of a local diffeomorphism, we can step forward to its rigorous definition.

\subsection{Étale Ring Maps}
We will present only those results about étale ring maps that are important for the construction of the Henselization. For other statements and some omitted proofs we refer to standard literature such as \cite{milne1980etale, Grothendieck67, raynaud1970anneaux} and of course \cite[\href{https://stacks.math.columbia.edu/tag/00U0}{Section 00U0}]{stacks-project}. 
\begin{definition}
    Given an $R$-algebra $S$ with the homomorphism $\phi: R \to S$, we call $S$ \emph{formally étale} if the following condition is satisfied:\\
    Suppose that $T$ is some $R$-algebra, $\idn \subseteq T$ some ideal with $\idn^2 =0$ and the following diagram of $R$-algebra maps commutes:
    \[\begin{tikzcd}
    S \arrow[r, "\bar{u}"] & T/\idn \\
    R \arrow[r] \arrow[u, "\phi"] & T \arrow[u, "\pi", swap]
    \end{tikzcd}
    \]
    Then there is a unique $R$-algebra morphism $u: S \to T$, which lifts $\bar{u}$, i.e. the following diagram also commutes:
    \[\begin{tikzcd}
    S \arrow[r, "\bar{u}"] \arrow[dr, dashed, "u"]& T/\idn \\
    R \arrow[r] \arrow[u, "\phi"] & T \arrow[u, "\pi", swap]
    \end{tikzcd}
    \]
\end{definition}
This property is known under the name \emph{infinitesimal lifting}. As we saw above, it reflects the definition of a local diffeomorphism inside of algebraic geometry. When dealing with a formally étale $S$, we will often refer to the map $\phi: R \to S$ as formally étale rather than to the $R$-algebra itself. 

To go from \textit{formally} étale to étale ring maps, we need to recall the notion of finitely presented algebras. It is evident that an $R$-algebra $S$ is always of the form $ S \cong R[x_i: i \in \mathcal{I}]/\ida$ for some index set $\mathcal{I}$ and an ideal $\ida \subseteq R[x_i : i \in \mathcal{I}]$. In practice we are often interested in a finite number of generators and a finitely generated ideal, hence we define:
\begin{definition}
	Let $R$ be a ring. We say an $R$-algebra $S$ is \emph{finitely presented} if it is of the form	$S \cong R[x_1,\dots,x_n]/(f_1,\dots, f_m)$ for some $f_i \in R[x_1,\dots, x_n]$, $i = 1,\dots,m$.
\end{definition}
Naturally, we have immediately the following fact about transitivity: If $S$ if finitely presented over $R$ and $T$ is finitely presented over $S$, then $T$ is finitely presented over $R$. Note that a localization of $R$ at one element, say $a \in R$, is finitely presented, since $R_a \cong R[t]/(at-1)$. Therefore a localization at finitely many elements is still finitely presented, but this does not have to be true for any multiplicative system.
\begin{definition}
	Let $S$ be an $R$-algebra. $S$ is called \emph{étale} if it is formally étale and finitely presented.
\end{definition}
\noindent The following lemma is a direct consequence of our definitions and observations.
\begin{lem} \label{etale_coposition}
	Let $S$ be an $R$-algebra and $S'$ an $S$-algebra. Assume that $R \to S$ and $S \to S'$ are (formally) étale, then the induced map $R \to S'$ is also (formally) étale. 
\end{lem}
Another important fact which again follows easily from chasing the correct diagram states that the property of being étale is stable under base change. Before proving this statement, we want to recall the definition and add a simple remark: Let $S$ be an $R$-algebra with $\phi: R \to S$ the corresponding map and let $R \to R'$ be any ring homomorphism. Then the \emph{base change of $\phi$ by $R \to R'$} is the ring map $R' \to R'\otimes_R S \eqqcolon S'$.
\[
\begin{tikzcd}
S \arrow[r] & S' \rlap{ $= R' \otimes_R S$}\\
R \arrow[r] \arrow[u, "\phi"] & R' \arrow[u, "\text{base change of } \phi", swap]
\end{tikzcd}
\]
Note that the explicit description of a base change is very natural when a presentation is given: We already saw that $S$, being an $R$-algebra, is of the form
\[
	S \cong R[x_i: i \in \mathcal{I}]/(f_j: j \in \mathcal{J}),
\]
for some index sets $\mathcal{I}, \mathcal{J}$ and polynomials $f_j \in R[x_i: i \in \mathcal{I}]$. Then, for the base change one has
\[
	R' \otimes_R S = R'[x_i: i \in \mathcal{I}]/(f_j': j \in \mathcal{J}),
\]
where each $f_j'$ is the image of $f_j$ under the map $R[x_i : i \in \mathcal{I}] \to R'[x_i : i \in \mathcal{I}]$ induced by the map $R \to R'$. In \cite[\href{https://stacks.math.columbia.edu/tag/05G3}{Tag 05G3}]{stacks-project} this fact is described as ``the key to understanding base change''. 
\begin{lem} \label{etale_basechange}
	Let $R \to S$ be étale and $R \to R'$ be arbitrary. Then $R' \to R' \otimes_R S$ is étale.
\end{lem}
\begin{cor} \label{etale_tensor}
	Let $R \to S$ and $R \to S'$ be étale. Then $R \to S \otimes_R S'$ is étale.
\end{cor}
\noindent
The following proposition has a more involved proof.
\begin{prop} \label{etale}
	Let $R$ be a ring, $S = R[x_1,\dots, x_n]_g/(f_1, \dots , f_n)$ for $g \in R[x_1,\dots,x_n]$ and $f_1, \dots, f_n \in R[x_1,\dots,x_n]_g$. If the image of the Jacobian determinant $\det(\frac{\partial f_j}{\partial x_i})_{1\leq i,j\leq n}$ is invertible in $S$, then $S$ is étale over $R$. \\
	Conversely, if $R\to S$ is étale, then there exists a presentation $S=R[x_1,\dots,x_n]/(f_1,\dots,f_n)$ such that the image of $\det(\frac{\partial f_j}{\partial x_i})_{1\leq i,j\leq n}$ is invertible in $S$. 
\end{prop}
For the proof we refer to standard literature \cite{raynaud1970anneaux}, \cite[Corollary 3.16]{milne1980etale}, to well-written lecture notes \cite{milneLEC}, \cite{hochsterLEC} and to the Stacks Project \cite[\href{https://stacks.math.columbia.edu/tag/00U0}{Section 00U0}]{stacks-project}.

\begin{definition}
	A finitely presented $R$-algebra $S$ is called \emph{standard étale} if it is of the form $S = R[t]_{g}/(f)$ for some polynomials $f, g \in R[t]$, such that $f$ is monic and its derivative $f'$ is invertible in $S$.
\end{definition}
\noindent
Note that by Proposition \ref{etale}, it follows that a standard étale algebra is indeed étale.

There exists a structure theorem of étale algebras, making sure that any étale algebra is \textit{locally} standard étale. In \cite[p. 120]{Grothendieck67} Grothendieck attributes this fact to Chevalley and so shall we.
\begin{thm}[Chevalley] \label{etale=standard_etale}
	Let $S$ be a finitely presented $R$-algebra. Then $S$ is étale over $R$ if and only if for every prime ideal $\mathfrak{q}$ of $S$ with contraction $\idp$ to $R$ there exist
	$b \in S \setminus \idq$ and $a \in R \setminus \idp$ such that $S_b$ is isomorphic to a standard étale algebra over $R_a$.
\end{thm}
The proof is an application of Zariski's main theorem, a form of which we already mentioned in Lemma \ref{37.8}. Also Nakayama's lemma and the primitive element theorem for separable field extensions play a role in the proof. In his lecture notes \cite{hochsterLEC} Hochster points out that ``additional trickery'' is required as well. Therefore the proof is lengthy and technical and shall not be provided here. We refer to \cite[pp. 51]{raynaud1970anneaux}, \cite[pp. 63]{iversen2006generic} \cite[Theorem 3.14]{milne1980etale} as well as \cite[\href{https://stacks.math.columbia.edu/tag/00UE}{Tag 00UE}]{stacks-project}, \cite[pp. 120]{Grothendieck67} and \cite[pp. 27]{hochsterLEC}. 

\subsection{Construction of the Henselization} \label{3.3}
We want to construct the Henselization of a local ring $(R,\idm,K)$ and consequently prove its existence and some desirable properties. First, we define the notion of étale neighborhoods like Milne in \cite{milne1980etale}:
\begin{definition}
	Let $(R,\idm,K)$ be local. A pair $(S,\idq)$ is called an \emph{étale neighborhood} of $R$ if $S$ is an étale $R$-algebra and $\idq$ is a prime of $S$ lying over $\idm$, such that the induced map between the residue fields $K = R/\idm \to S_{\idq}/\idq S_{\idq}$ is an isomorphism.
\end{definition}
\noindent
In order to save notation in our setting, it is more useful to work \textit{locally} and to use the notion of pointed étale extensions, as does Hochster in his lecture notes \cite{hochsterLEC}:
\begin{definition}
	A local ring $T$ is called \emph{pointed étale extension} of $(R,\idm,K)$ if $T = S_\idq$ for some étale neighborhood $(S,\idq)$.
\end{definition}
Before stating and proving a theorem which connects étale ring maps and Henselian rings, we need to state the following lemma.
\begin{lem} \label{pointed_etale_unique}
    Assume $T$ and $T'$ are pointed étale extensions of $R$. Then there is at most one local $R$-algebra homomorphism from $T$ to $T'$.
\end{lem}
The proof of this lemma requires the study the multiplication map, given by the linear extension of $\mu: S \otimes_RS \twoheadrightarrow S$ sending $s \otimes s' \mapsto ss'$, and its kernel $\mathfrak{a} \coloneqq \ker(\mu)$. We omit the details and refer to \cite[pp. 45]{hochsterLEC}, \cite[p. 71]{iversen2006generic}, \cite[p. 36]{milne1980etale} and \cite[Section 4]{milneLEC}. Now we are ready for the central theorem of this section. Its statement and proof can be found in the references just mentioned, however we shall reprove it again following the notes of Hochster.
\begin{thm} \label{henselian}
	Let $(R, \idm, K)$ a be local ring. The following conditions are equivalent:
	{\setlength{\parindent}{10pt}
    \begin{enumerate}[label=\textit{(\arabic*)},wide]
		\item $R$ is Henselian.
		\item If $f \in R[t]$ is a monic polynomial whose reduction \text{\normalfont mod} $\idm$, $\bar{f} \in K[t]$, has a simple root $\lambda \in K$, then there exists an element $r \in R$ such that $r \equiv \lambda \mod \idm$ and $f(r) = 0$.
		\item If $R \to T$ is a pointed étale extension, then $R \cong T
		$.
		\item If $f_1,\dots, f_n \in R[x_1,\dots, x_n]$ are $n$ polynomials in $n$ variables whose images $\overline{f_j} \mod \idm$ vanish simultaneously at $(\lambda_1, \dots , \lambda_n) \in K^n$ and the Jacobian determinant $\det (\frac{\partial f_j}{\partial x_i}) $ does not vanish$\mod \idm$ at $x_1 = \lambda_1, \dots , x_n = \lambda_n$, then there are unique elements $r_1, \dots , r_n \in R$ such
		that for all $i$, we have $r_i \equiv \lambda_i \mod \idm$ and $f_j(r_1,\dots, r_n) = 0, 1 \leq j \leq n$.
	\end{enumerate}
	}
\end{thm}
This theorem gives a deep insight into Henselian rings. In particular, the equivalence of conditions $(1)$ and $(2)$ implies that it suffices to lift only simple roots in order to be able to lift coprime factorizations. Applying this for the ring of algebraic power series and Theorem \ref{hensel_aps0} we obtain a proof for Theorem \ref{hensel_aps1}. We see that Henselian rings are connected to the theory of étale ring maps via condition $(3)$. Moreover, $(4)$ is a multidimensional version of Hensel's lemma for $n$ polynomials and $n$ variables; if the $f_i$'s were also allowed to be power series, one would recognize the implicit function theorem. The equivalence $(1) \Leftrightarrow (4)$ states that a ring is Henselian if and only if the algebraic version of this analytic theorem holds in this ring. 
\begin{proof}
	We will show that $(1) \Rightarrow (2) \Rightarrow (3) \Rightarrow (4) \Rightarrow (1)$. 
	
	$(1) \Rightarrow (2)$: Suppose $R$ is Henselian and we have a monic $f \in R[t]$ such that $\bar{f}$ has a simple root $\lambda \in K$. We may factor $\bar{f}(t) = (t-\lambda)\bar{g}(t)$ for some $\bar{g} \in K[t]$ with $\bar{g}(\lambda) \neq 0$. The polynomials $t-\lambda$ and $\bar{g}(t)$ are relatively prime. Using the assumption we find a lifting of the factorization to $f(t) = (t-r)g(t)$ for some $r \equiv \lambda \mod \idm$ and $g \in R[t]$. Clearly $f(r) = 0$.
	
	$(2) \Rightarrow (3)$: Let $\phi: R \to T$ be a pointed étale extension, so a localization of an étale neighborhood. By Theorem \ref{etale=standard_etale} it follows that the étale neighborhood is locally standard étale, hence we may write $T \cong (R[t]_{g}/(f))_\idq$ for a prime ideal $\idq \subseteq R[t]_{g}/(f)$ lying over $\idm$ and $g, f \in R[t]$ such that $f$ is monic and $f'$ is invertible in $T$. Denoting by $\lambda$ the image of $t \in T$ in $K$, it follows that $\bar{f}(\lambda) = 0$. Moreover, because $f'$ is invertible in $T$, we must have that $\bar{f}'(\lambda) \neq 0$, hence $\lambda$ is a simple root of $\bar{f}$. Using $(2)$, we can find an $r \in R$ for which $f(r) = 0$. Therefore there exists an $h \in R[t]$ such that $f(t) = (t-r)h(t)$ and $h$ is invertible in $T$, because $\lambda$ is a simple root of $\bar{f}$. It follows that 
	\[
		T \cong (R[t]_{g}/(f))_\idq \cong  \bigl(R[t]_{g}/((t-r)h(t))\bigr)_\idq \cong (R[t]_{g}/(t-r))_\idq \cong R_{\phi^{-1}(\idq)}.
	\]
	However, since $\phi$ and $R$ are local, we have that $T \cong R_{\phi^{-1}(\idq)} \cong R$, what was to be shown.
	
	$(3) \Rightarrow (4)$: Assume we have a system of equations $f_1, \dots, f_n \in R[x_1,\dots,x_n]$ like in $(4)$ with $(\lambda_1,\dots,\lambda_n) \in K^n$ solution of $(\bar{f_1},\dots,\bar{f_n})=0$ and suppose that $(3)$ holds. Let $Q$ be the kernel of $\pi': R[x_1,\dots,x_n] \to K$, where we choose $\pi'$ such that $\pi'(x_i) = \lambda_i$. By Proposition \ref{etale} and the assumption on the Jacobian of the $f_1,\dots,f_n$ in $(4)$, we have that $T \coloneqq R[x_1,\dots,x_n]_Q/(f_1,\dots,f_n) \cong (R[x_1,\dots,x_n]/(f_1,\dots,f_n))_{\bar{Q}}$ is a pointed étale extension of $R$. Because of $(3)$ we must have that $R \cong T$. However, solving the equations $f_1,\dots,f_n$ and lifting the $\lambda_i$'s is equivalent to giving an $R$-algebra map $R[x_1,\dots,x_n]/(f_1,\dots,f_n) \to R$ such that under the composite $R[x_1,\dots,x_n]/(f_1,\dots,f_n) \to R \twoheadrightarrow K$ the elements $x_i$ map to $\lambda_i$. This is in turn equivalent to giving a map that sends $Q$ to $\idm$, hence giving a local $R$-algebra map $T \to R$. But we have that $R \cong T$, hence the local map exists and is unique by Lemma \ref{pointed_etale_unique}. It provides us with a unique solution to the equations. 
	
	$(4) \Rightarrow (1)$: Let $f = t^n + c_{n-1}t^{n-1} + \cdots + c_1t+c_0 \in R[t]$ be a monic polynomial of degree $n$ and suppose that we have a factorization $\bar{f} = \bar{g}\bar{h}$ for some monic coprime polynomials $\bar{g},\bar{h} \in K[t]$ of degrees $d$ and $e$ respectively. Let $\bar{g} = \sum_{i=0}^d \alpha_i t^i$ and $\bar{h} = \sum_{i=0}^e \beta_i t^i$ for some $\alpha_i,\beta_i \in K$ and $\alpha_d = \beta_e =1$. We seek a lifting of the factorization to $f = gh$ for monic polynomials $g,h \in R[t]$. Let the coefficients of $g$ and $h$ be unknowns $y_0,\dots, y_{d-1}$ and $z_0,\dots,z_{e-1}$, henceforth we want to solve the equation
	\[
		t^n + c_{n-1}t^{n-1} + \cdots + c_1t+c_0 = (t^d+ y_{d-1}t^{d-1} + \cdots + y_1t+ y_0)(t^e+ z_{e-1}t^{e-1} + \cdots + z_1t+ z_0),
	\]
	for the unknowns over $R$ such that the residue classes of the polynomials $g$ and $h$ agree with $\bar{g}$ and $\bar{h}$. Comparing coefficients leads to a system of $n=d+e$ polynomial equations in as many variables:
	\[
	 \begin{cases}
	 y_0 z_0 & = c_0,\\
	 y_0z_1 + y_1z_0 & = c_1,\\
	 \vdots\\
	 y_{d-1}z_e + y_d z_{e-1} &  = c_{n-1}. \\
	 \end{cases}
	\]
	This system has a solution mod $\idm$ coming from the factorization $\bar{f} = \bar{g}\bar{h}$ given by $\alpha_0,\dots,\alpha_{d-1}, \beta_0, \dots, \beta_{e-1} \eqqcolon (\alpha, \beta)$. In order to use $(4)$ and to lift this solution to $R$ we have to verify that the Jacobian determinant of this system of equations does not vanish, i.e. that the matrix 
	\begin{align*}
		J(y,z) \coloneqq 
		\begin{pmatrix}
			z_0    & z_1    & z_2    & \cdots & z_{e-1} & 1       & 0       & \cdots  & 0 \\
			0      & z_0    & z_1    & \cdots & z_{e-2} & z_{e-1} & 1       & \cdots  & 0 \\
			\vdots &        & \ddots & 	      &         &         & \ddots  &         & \vdots \\
			0      & \cdots & 0      & z_0    & z_1     & \cdots  & z_{e-2} & z_{e-1} & 1  \\
			y_0    & y_1    & y_2    & \cdots & y_{d-1} & 1       & 0       & \cdots  & 0 \\
			0      & y_0    & y_1    & \cdots & y_{d-2} & y_{d-1} & 1       & \cdots  & 0 \\
			\vdots &        & \ddots & 	      &         &         & \ddots  &         & \vdots \\
			0      & \cdots & 0      & y_0    & y_1     & \cdots  & y_{d-2} & y_{d-1} & 1  
		\end{pmatrix}
	\end{align*}
	is invertible at $(y,z) = (\alpha,\beta)$. However, $J(\alpha,\beta)$ is the (transpose of the) Sylvester matrix of the polynomials $\bar{g}$ and $\bar{h}$. Since the polynomials are relatively prime by assumption, we obtain that $J(\alpha,\beta)$ is invertible. This shows that the assumptions of $(4)$ are satisfied and hence we find a unique solution for the unknowns $y_0,\dots,y_{d-1},z_0,\dots,z_{e-1}$. This gives the unique factorization $f = gh$ we were looking for.
\end{proof}
Having in mind that $(1) \Leftrightarrow (3)$ in this theorem, we come back to our goal of constructing the Henselization. We see that it may be a good idea to combine all possible pointed étale extensions of $R$ into one large ring. If we can do this rigorously, then we might argue that this ring does not have any proper pointed étale extensions anymore, which will mean that it will be Henselian. Finally, we might be able to verify the universal property of the Henselization and conclude that we indeed found the correct object. Let us start executing this plan.\\

Given a local ring $R$, we wish to define a \textit{set} of pointed étale algebras of $R$, say $\mathcal{R}$, that contains exactly one representative from each isomorphism class of pointed étale extensions. It is not trivial that $\mathcal{R}$ is a set, since it might turn out ``too large''. However, the following result bounds the cardinality of a pointed étale extension from above and allows us to define $\mathcal{R}$ properly.
\begin{lem}
	Let $R$ be a local ring and $T$ a pointed étale extension. Then $T$ is finite if $R$ is finite. In the other case, the cardinalities of $R$ and $T$ agree.
\end{lem}
\begin{proof}
	By definition, $T = S_\idq$ for some prime $\idq \subseteq S$ and an étale $R$-algebra $S$. Since $S$ is finitely presented over $R$, we have $|S| \leq |R|^n$ for some $n \in \mathbb{N}$, where $|\cdot|$ denotes cardinality. The localization is parametrized by pairs in $(S \setminus \idq) \times S$ and therefore $|S_\idq| \leq |S|^2$. We have
	\[
		|R| \leq |T| = |S_\idq| \leq |R|^{2n},
	\]
	proving the assertion. 
\end{proof}
Now, from the axiom of choice, it follows that the \textit{set} $\mathcal{R}$ exists, since it is a subset of the \textit{set} of all ring structures on a set with similar cardinality as $R$. Let $\mathcal{A}$ be an index set of $\mathcal{R}$, whereby \emph{index set} means that each $i \in \mathcal{A}$ corresponds bijectively to a $T_i \in \mathcal{R}$ and we can write therefore $\mathcal{R} = (T_i)_{i \in \mathcal{A}}.$

\begin{prop} \label{directed_set}
	For $i, j \in \mathcal{A}$ define $i \leq j$ if and only if there exists a local $R$-algebra map $\phi_{i,j} : T_i \to T_j$. Then $(\mathcal{A}, \leq)$ is a directed set.
\end{prop}
\begin{proof}
	Obviously, $\mathcal{A}$ is not empty since $\mathcal{R}$ contains $R$. Clearly, $\leq$ is reflexive, as one always has the identity map $\text{id}: T_i \to T_i$ for all $i$. Moreover, if we have $\phi: T_i \to T_j$ and $\psi: T_j \to T_k$ both local $R$-algebra maps then $\psi \circ \phi: T_i \to T_k$ is a local $R$-algebra map. This implies that $\leq$ is transitive. Finally, we have to prove that for any two $T_i, T_j \in \mathcal{R}$, there exist $T_k \in \mathcal{R}$ and two local $R$-algebra maps $T_i \to T_k$ and $T_j \to T_k$. As $T_i,T_j$ are pointed étale, they are localizations of some étale $R$-algebras $S_i, S_j$. By Corollary \ref{etale_tensor} we immediately have that $R \to S_i \otimes_R S_j$ is étale. Consider the composite map
	\[
		R \to S_i \otimes_R S_j \twoheadrightarrow K \otimes_K K \xrightarrow{\cong} K,
	\]
	which sends $r \mapsto r \cdot (1_{S_i} \otimes 1_{S_j}) \mapsto \bar{r} \cdot (1_K \otimes 1_K) \mapsto \bar{r}$ and is thus precisely the quotient map $R \twoheadrightarrow R/\idm \cong K$. It follows that by letting $Q$ be the kernel of the map $S_i \otimes_R S_j \twoheadrightarrow K \otimes_K K$, we must have that $R \to (S_i \otimes_R S_j)_Q$ is local. The residue class field of $(S_i \otimes_R S_j)_Q$ is $K$. Set $T_k = (S_i \otimes_R S_j)_Q $ which is now by definition a pointed étale extension of $R$ and we have maps $T_i \to T_k$ and $T_j \to T_k$. This shows the existence of a $k \in \mathcal{A}$ for given $i,j \in \mathcal{A}$ such that $i,j\leq k$ and finishes the proof. 
\end{proof}
This proposition shows that $(\mathcal{R}, \{\phi_{i,j}: T_i \to T_j\}_{i,j\in \mathcal{A}, i\leq j})$ forms a direct system of rings. Note that because of Lemma \ref{pointed_etale_unique}, we know that the $\phi_{i,j}$'s are actually unique, justifying that the construction is canonical. The fact that $\mathcal{R}$ together with these maps forms a direct system of rings allows us to define the direct limit:
\begin{definition} \label{e}
	For a local ring $(R, \idm, K)$ we denote $\e{R} \coloneqq \dirlim_{T \in \mathcal{R}} T$.
\end{definition}
Given a local ring $R$, we combine all pointed étale extensions of it to the ring $\e{R}$ in a rigorous way using the direct limit in the definition above. Therefore, it is natural to expect that $\e{R}$ does not have any proper pointed étale extensions anymore, which means that $\e{R}$ is Henselian by Theorem \ref{henselian}. It is also intuitively clear that $\e{R}$ the ``smallest'' extension of $R$ that admits this property. We prove both statements below using ideas from \cite{iversen2006generic} and \cite{hochsterLEC}.
\begin{lem}
	Let $(R,\idm,K)$ be a local ring. Then $\e{R}$ is local with maximal ideal $\idm R$ and residue field $K$. Moreover, $\e{R}$ is Henselian.
\end{lem}
\begin{proof}
	Locality, the statement about the maximal ideal and the condition on the residue field follow by construction, since every pointed étale $R$-algebra $T$ is local with maximal ideal $\idm T$ and residue field $K$.
	
	By Theorem \ref{henselian} we only have to check the lifting of simple roots in order to verify the Henselian property. Let $f \in \e{R}[t]$ be monic and $\lambda \in K$ a simple root of $\bar{f} \in K[t]$. Since $\e{R} = \dirlim_{T \in \mathcal{R}} T$, there exists some pointed étale $R$-algebra $T$ such that all coefficients of $f$ lie in $T$. We define $T' \coloneqq (T[t]/(f))_\idq$, where $\idq \coloneqq (\bar{t}-\lambda)$. The residue field of $T'$ is $K$ and because $\lambda$ is a simple root, it follows that $f'$ is invertible in $T'$ and therefore $T'$ is a pointed étale extension of $R$ by Lemma \ref{etale_coposition} and Proposition \ref{etale}. However, $f$ has a root in $T'$ and it lifts $\lambda$. This gives rise to an element $r \in \e{R}$ such that $f(r)=0$ and $\bar{r} = \lambda$.
\end{proof}
\begin{thm} \label{etale=limit}
	Let $(R, \idm, K )$ be a local ring. The Henselization of $R$ is given by the direct limit as in Definition \ref{e}: $\hen{R} = \e{R}$.
\end{thm}
\begin{proof}
	We will verify the universal property. From the lemma above we already have that $\e{R}$ is local and Henselian. Let $\psi: R \to H$ be a local map from $R$ to a Henselian ring $(H, \idm_H, L)$. To show that this map factors uniquely through $\e{R}$, it suffices to show that it factors uniquely through every $(T, \idq T, K)$, where $T = S_\idq$ is a pointed étale extension of $R$. Consider the commutative diagram of the base change:
	\[
	\begin{tikzcd}
	R \arrow[d, "\psi"] \arrow[r, "\text{étale}"] & S \arrow[d] \\
	H \arrow[r] & S \otimes_R H 
	\end{tikzcd}
	\]
	Since $R \to S$ is étale, we obtain by Lemma \ref{etale_basechange} that $H \to S \otimes_R H$ is also étale. Moreover, there exists a canonical map $S \otimes_R H \to K \otimes_K L \cong L$. Denote its kernel by $Q$. It follows that $H \to (S \otimes_R H)_Q$ is a localization of an étale extension. Since $L \cong K \otimes_K L$, we obtain that the residue fields agree and hence this extension is pointed étale. But $H$ is Henselian, hence $H \cong (S \otimes_R H)_Q$ by $(1) \Rightarrow (3)$ of Theorem \ref{henselian} and therefore we found a local map $\phi: S \to (S \otimes_R H)_Q \cong H$, the map we were looking for:
		\[
	\begin{tikzcd}
	R \arrow[d, "\psi"] \arrow[r, "\text{étale}"] & S \arrow[d] \\
	H \arrow[r, "\text{étale}"] \arrow[dr, <->,swap, "\cong"]& S \otimes_R H \arrow[d] \\
	& (S \otimes_R H)_Q 
	\end{tikzcd}
	\]
	Finally, because $H$ is pointed étale over itself as well as over $(S \otimes_R H)_Q$, we obtain that this map is unique by Lemma \ref{pointed_etale_unique}.
\end{proof}
The characterization of the Henselization as a direct limit of pointed étale extensions not only proves its existence, but also led to remarkable mathematical discoveries in this area in the second half of the last century. The theory of approximation rings and consequently the algebraic version of Artin's Approximation \cite{artinalgebraic} use exactly this fact (amongst other). This and other related results are contained in the recent survey \cite{Hauser17} by Hauser. The Henselization of non-local rings with respect to ideals is investigated in \cite{Greco1969HenselizationOA}. An exposition of various versions of the Henselian property can be found in \cite{Ribenboim1985EquivalentFO}. The connection of Henselian rings to rings satisfying Weierstrass preparation theorem was first established by Lafon in 1967 \cite{lafon2}. Probably the most explicit application of the Theorems \ref{etale=standard_etale} and \ref{etale=limit} was found by Denef and Lipshitz in 1985 and we shall explain their ideas in the next section.

\section{Explicit Implications} \label{sec:4}
\subsection{Proof of Theorem \ref{DL}}
\begin{thm}[Denef \& Lipshitz] \label{algebraic=etale-algebraic}
	Let $f \in K\aps{x}$ be an algebraic power series. Then there exist an étale-algebraic power series $h$ and polynomials $a_i, b_j \in K[x]$ for $0 \leq i \leq r$, $0 \leq j \leq s$, $r,s \in \mathbb{N}$, where $b_0(0) \neq 0$, such that
	\begin{align} 
	f = \frac{a_0 + a_1 h+ \dots + a_rh^r}{b_0+ b_1h + \dots  + b_sh^s}.
	\end{align}  
\end{thm}
\begin{proof}
	Set $R \coloneqq K[x]_{(x)}$. We have seen that $K\aps{x} = \hen{R} = \dirlim_{T \in \mathcal{R}} T$, where the limit is taken over all pointed étale extensions up to isomorphism. It follows that there exists a ring $T \subseteq K\aps{x}$ which is a pointed étale extension of $R$ and which contains $f$. Hence, $T = S_\idq$ for an étale $R$-algebra $S$ and a prime ideal $\idq \subseteq S$ lying over $\idm \subseteq R$. We know furthermore by Theorem \ref{etale=standard_etale} that $S$ is locally standard étale over $R$; since $R$ is local, this means that we have an isomorphism
	\[
		\alpha: S_b \xrightarrow{\cong} R[t]_{g}/(p)
	\]
	for some $b \in S \setminus \idq$, $g \in R[t]$ and $p \in R[t]$ monic such that its derivative $p'$ is invertible in $S_b$. Localizing in $\idq$ and $\alpha(\idq)$, respectively, yields $T \cong \left(R[t]/(\tilde{P}) \right)_{\alpha(\idq)}$ for some $\tilde{P} \in R[t]$ such that $\tilde{P}' \not \in \alpha(\idq)$. We can rephrase the isomorphism above as $T \cong (R[\tilde{h}])_{{\alpha}(\idq)}$, where $\tilde{h} \in \widehat{R} = K\ps{x}$ is an algebraic element over $R$ whose minimal polynomial is exactly $\tilde{P}$. Because $\tilde{P}' \not \in \tilde{\alpha}(\idq)$, we have $\partial_t\tilde{P}(0,\tilde{h}(0)) \neq 0$. Now, any element $f \in (R[\tilde{h}])_{{\alpha}(\idq)} \cong T$ is of the form $a/b$, for $a, b\in R[\tilde{h}]$ and $b \not \in \tilde{\alpha}(\idq)$, hence we have 
	\begin{align*} 
		f(x) = \frac{\tilde{a}(x,\tilde{h}(x))}{\tilde{b}(x,\tilde{h}(x))},
	\end{align*}
	for $\tilde{a}, \tilde{b} \in K[x]_{(x)}[t]$ such that $\tilde{b}(0,\tilde{h}(0)) \neq 0$. Finally, to achieve the condition $h(0) = 0$ as in the definition of étale-algebraic, we define $h(x) = \tilde{h}(x) - \tilde{h}(0)$. It is easy to verify that the derivative of the minimal polynomial $P(x,t)$ of $h$ does not vanish at the origin, $\partial_t P(0,0) = \partial_t\tilde{P}(0,\tilde{h}(0)) \neq 0$, and that we have again 
	\[
		f = \frac{a_0 + a_1 h+ \dots + a_rh^r}{b_0+ b_1h + \dots  + b_sh^s},
	\]
	for polynomials $a_i, b_j \in K[x]$ such that $b_0(0) = \tilde{b}(0,\tilde{h}(0)) \neq 0$.
\end{proof}
Now we can finally prove Theorem \ref{DL}.
\begin{proof}[Proof of Theorem \ref{DL}]
	Using the theorem above it follows that there exist $r,s \in \mathbb{N}$ and $a_i(x), b_j(x) \in K[x]$ for $0 \leq i \leq r$, $0 \leq j \leq s$ with $b_0(0) \neq 0$ such that
	\begin{align} \label{etale_alg}
	f(x) = \frac{a_0(x) + a_1(x) h(x)+ \dots + a_r(x)h(x)^r}{b_0(x)+ b_1(x)h(x) + \dots  + b_s(x)h(x)^s},
	\end{align}  
	where $h(x) \in K\aps{x}$ is étale-algebraic. Define 
	\begin{align*}
	W(x,t) \coloneqq \frac{a_0(x) + a_1(x) t+ \dots + a_r(x)t^r}{b_0(x)+ b_1(x)t + \dots  + b_s(x)t^s} \in K[x,t]_{(x,t)},
	\end{align*}
	and let 
	\begin{align*}
	R(x,t) \coloneqq  W(xt,t) t \frac{\partial_t P(xt,t)}{P(xt,t)}.
	\end{align*}
	Using Lemma \ref{diagonallemma} and the same computation as in equation (\ref{diagonalformula}) we verify that we found the correct rational function:
	\[
		\mathcal{D}(R(x,t)) = W(x,h(x)) = h(x). \qedhere
	\]
\end{proof}

\subsection{Codes of Algebraic Power Series}
Finally, we introduce a new result which can be seen as a corollary of Theorem \ref{algebraic=etale-algebraic}. First, we remark that the following fact was explained in \cite[pp. 88]{artinmazur} and became later known under the name Artin-Mazur lemma, see \cite{bochnak2013real, AlMoRa92}. In \cite[Proposition 9.3]{Rond2018} a more general version of the statement, allowing for $K$ to be any complete normal local domain (with appropriate changes to the assumptions), is presented.
\begin{thm}[Artin \& Mazur] \label{am}
Let $f \in K \langle x_1,\dots, x_n \rangle = K \langle x \rangle$ be an algebraic power series with $f(0)=0$. Then there exist $k \in \mathbb{N}$ and a vector of $k$ polynomials $P(x,y_1,\dots,y_k) \in K[x][y_1,\dots,y_k]^k$ with the following properties:
{\setlength{\parindent}{10pt}
\begin{enumerate}[label=\textit{(\arabic*)},wide]
		\item $P(x,f,h_2,\dots,h_k) = 0$ for algebraic power series $h_2,\dots,h_k \in K\aps{x}$ with $h_i(0)=0$ for $i = 2,\dots,k$.
		\item The Jacobian matrix $J_P(x,y_1,\dots,y_k)$ of $P(x,y_1,\dots,y_n)$ with respect to the variables $y_1,\dots, y_k$ at $x=y=0$ is invertible: $J_P(0,0) \in \GL_k(K)$.
	\end{enumerate}
	}
\end{thm}
In other words, given $f \in K\aps{x}$, one can find $k-1$ algebraic power series $h_2,\dots,h_k \in K\aps{x}$ and a $k$-dimensional vector of polynomials $P(x,y_1,\dots,y_k) \in K[x,y]^k$, such that $P(x,f(x),h_2(x),\dots,h_k(x)) =0$ and the Jacobian of $P(x,y)$ with respect to $y$ at $x=y=0$ is invertible. Similarly to Theorem \ref{algebraic=etale-algebraic}, this implies that one can repair the problem of an algebraic power series of not being étale-algebraic, now by appending $k-1$ new power series and considering the $k$-dimensional analogue of the definition of étale-algebraicity. This polynomial vector $P(x,y) \in K[x,y]^k$ is referred to as \textit{a (mother) code of the algebraic series} $f$ in \cite{HerwigAlonso2014EncodingAP, Hauser17, AlMoRa92}. The authors Alonso, Castro-Jimenez and Hauser of the first reference point out that ``The advantage of this code in comparison with taking the minimal polynomial lies in the fact that the latter determines the algebraic series only up to
conjugation, so that extra information is necessary to specify the series, typically a sufficiently
high truncation of the Taylor expansion. In contrast, the polynomial code determines the series completely and is easy to handle algebraically''.

With the help of the theorem of Denef and Lipshitz we can improve on the Artin-Mazur lemma, proving that it is always possible to choose $k=2$. Note that since Theorem \ref{DL} is known to hold in more generality, it is also natural that also the general version by Rond can be covered and improved by our approach.
\begin{thm} \label{code}
	Let $f \in K \langle x_1,\dots, x_n \rangle = K \langle x \rangle$ be an algebraic power series with $f(0)=0$. Then there exists a vector of two polynomials $P(x,y_1,y_2) \in K[x][y_1,y_2]^2$ with the following properties:
{\setlength{\parindent}{10pt}
\begin{enumerate}[label=\textit{(\arabic*)},wide]
		\item $P(x,f,h) = 0$ for some étale-algebraic power series $h \in K\aps{x}$.
		\item The Jacobian matrix $J_P(x,y_1,y_2)$ of $P(x,y_1,y_2)$ with respect to $y_1$ and $y_2$ at $0$ is invertible: $J_P(0,0,0) \in \GL_2(K)$.
	\end{enumerate}
	}
\end{thm}
Note that in the two-dimensional square matrix $J_P(0,0,0)$, the first $0$ means setting the variables $x_1,\dots,x_n$ all to $0$ in $J_P(x,y_1,y_2)$, whereas the other two zeros are both one-dimensional and advert to $y_1$ and $y_2$.

\begin{proof}
	Let $Q(x,y_1)$ be the minimal polynomial of $f$. If $\partial_{y_1}Q(0,0) \neq 0$ then we can simply choose $P(x,y_1,y_2) = (Q(x,y_1), y_2)$ and the assertion follows in this case.
	
	We are left with the more challenging case $\partial_{y_1}Q(0,0) = 0$. By Theorem \ref{algebraic=etale-algebraic}, we may write for some étale-algebraic power series $h \in K\aps{x}$
	\begin{align} \label{4.2}
	f = \frac{a_0 + a_1 h+ \dots + a_rh^r}{b_0+ b_1h + \dots  + b_sh^s},
	\end{align}  
	for $r,s \in \mathbb{N}$ and $a_i(x), b_j(x) \in K[x]$ with $0 \leq i \leq r$, $0 \leq j \leq s$ and $b_0(0) \neq 0$.
	Define the polynomials
	\begin{align*}
	T_1(x,y_2) & \coloneqq a_0(x) + a_1(x) y_2+ \dots + a_r(x)y_2^r,\\ 
	T_2(x,y_2) & \coloneqq b_0(x)+ b_1(x)y_2 + \dots  + b_s(x)y_2^s, 
	\end{align*}
	and get the relationship $T_1(x,h(x)) = f(x)T_2(x,h(x))$ from identity (\ref{4.2}). Let $S(x,y_2)$ be the minimal polynomial of the étale-algebraic $h(x)$, so that $\partial_{y_2}S(0,0) \neq 0$. Now we put
	\begin{align*}
	P(x,y_1,y_2) \coloneqq \begin{pmatrix}
	y_1 T_2(x,y_2) - T_1(x,y_2) \\
	S(x,y_2)
	\end{pmatrix}.
	\end{align*}
	A simple computation confirms that this choice of $P$ satisfies all required properties:
	\begin{align*}
	P(x,f(x),h(x)) & = 0 \hspace{0.5cm} \text{and}\\
	J_P(0,0,0) = \begin{pmatrix}
	T_2(x,y_2) & * \\
	0  & \partial_{y_2}S(x,y_2)
	\end{pmatrix} &  \Bigg|_{(0,0,0)} = \begin{pmatrix}
	T_2(0,0) & * \\
	0  & \partial_{y_2}S(0,0)
	\end{pmatrix}.
	\end{align*}
	Clearly, $\det(J_P(0,0,0)) =  T_2(0,0) \partial_{y_2}S(0,0) \neq 0$, because both factors are different from $0$.
\end{proof}

\phantomsection
\subsection*{Acknowledgments}
First of all, the author wants to thank Herwig Hauser for the supervision and correction of the master's thesis version of this work. The author is also grateful to Christoper Chiu and Giancarlo Castellano for patiently explaining to him central concepts related to this work, reading its early versions and suggesting improvements. Moreover, we thank the anonymous referee for helpful comments and finally Alin Bostan and Kilian Raschel for organizing the wonderful conference ``Transient Transcendence In Transylvania'' in 2019.   

\addcontentsline{toc}{section}{Bibliography}

\bibliographystyle{alphaabbr}
\bibliography{bib}

\begin{thebibliography}{ACJH18}

\bibitem[AB13]{AdAdDe13}
B.~Adamczewski and J.~P. Bell.
\newblock Diagonalization and rationalization of algebraic Laurent series.
\newblock {\em Annales scientifiques de l'\'Ecole Normale Sup\'erieure}, Ser.
  4, 46(6):963--1004, 2013.

\bibitem[ABD19]{AdAdDe16}
B.~Adamczewski, J.~P. Bell, and E.~Delaygue.
\newblock Algebraic independence of {$G$}-functions and congruences ``\`a la
  {L}ucas''.
\newblock {\em Ann. Sci. \'{E}c. Norm. Sup\'{e}r. (4)}, 52(3):515--559, 2019.

\bibitem[ACJH18]{HerwigAlonso2014EncodingAP}
M.~E. Alonso, F.~J. Castro-Jim\'{e}nez, and H.~Hauser.
\newblock Encoding algebraic power series.
\newblock {\em Found. Comput. Math.}, 18(3):789--833, 2018.

\bibitem[AM65]{artinmazur}
M.~Artin and B.~Mazur.
\newblock On periodic points.
\newblock {\em Ann. of Math. (2)}, 81:82--99, 1965.

\bibitem[AMR92]{AlMoRa92}
M.~E. Alonso, T.~Mora, and M.~Raimondo.
\newblock A computational model for algebraic power series.
\newblock {\em J. Pure Appl. Algebra}, 77(1):1--38, 1992.

\bibitem[Art69]{artinalgebraic}
M.~Artin.
\newblock Algebraic approximation of structures over complete local rings.
\newblock {\em Inst. Hautes \'{E}tudes Sci. Publ. Math.}, (36):23--58, 1969.

\bibitem[Azu51]{azumaya1951}
G.~Azumaya.
\newblock On maximally central algebras.
\newblock {\em Nagoya Math. J.}, 2:119--150, 1951.

\bibitem[BCR98]{bochnak2013real}
J.~Bochnak, M.~Coste, and M.-F. Roy.
\newblock {\em Real algebraic geometry}, volume~36 of {\em Ergebnisse der
  Mathematik und ihrer Grenzgebiete (3)}.
\newblock Springer-Verlag, Berlin, 1998.
\newblock Translated from the 1987 French original, Revised by the authors.

\bibitem[BDS17]{BoDuSa17}
A.~Bostan, L.~Dumont, and B.~Salvy.
\newblock Algebraic diagonals and walks: algorithms, bounds, complexity.
\newblock {\em J. Symbolic Comput.}, 83:68--92, 2017.

\bibitem[BLS17]{BoLaSa17}
A.~Bostan, P.~Lairez, and B.~Salvy.
\newblock {Multiple binomial sums}.
\newblock {\em {Journal of Symbolic Computation}}, 80(2):351--386, 2017.

\bibitem[Chr15]{Christol}
G.~Christol.
\newblock Diagonals of rational fractions.
\newblock {\em Eur. Math. Soc. Newsl.}, (97):37--43, 2015.

\bibitem[Del84]{Deligne1984}
P.~Deligne.
\newblock Int\'{e}gration sur un cycle \'{e}vanescent.
\newblock {\em Invent. Math.}, 76(1):129--143, 1984.

\bibitem[DL87]{DeLi87}
J.~Denef and L.~Lipshitz.
\newblock Algebraic power series and diagonals.
\newblock {\em J. Number Theory}, 26(1):46--67, 1987.

\bibitem[Eis95]{eisenbud1995commutative}
D.~Eisenbud.
\newblock {\em Commutative algebra}, volume 150 of {\em Graduate Texts in
  Mathematics}.
\newblock Springer-Verlag, New York, 1995.
\newblock With a view toward algebraic geometry.

\bibitem[Fur67]{Furstenberg67}
H.~Furstenberg.
\newblock Algebraic functions over finite fields.
\newblock {\em J. Algebra}, 7:271--277, 1967.

\bibitem[Gre69]{Greco1969HenselizationOA}
S.~Greco.
\newblock Henselization of a ring with respect to an ideal.
\newblock {\em Trans. Amer. Math. Soc.}, 144:43--65, 1969.

\bibitem[Gro67]{Grothendieck67}
A.~Grothendieck.
\newblock \'{E}l\'{e}ments de g\'{e}om\'{e}trie alg\'{e}brique. {IV}. \'{E}tude
  locale des sch\'{e}mas et des morphismes de sch\'{e}mas {IV}.
\newblock {\em Inst. Hautes \'{E}tudes Sci. Publ. Math.}, (32):361, 1967.

\bibitem[Har88]{Harase1988}
T.~Harase.
\newblock Algebraic elements in formal power series rings.
\newblock {\em Israel J. Math.}, 63(3):281--288, 1988.

\bibitem[Hau17]{Hauser17}
H.~Hauser.
\newblock The classical {A}rtin approximation theorems.
\newblock {\em Bull. Amer. Math. Soc. (N.S.)}, 54(4):595--633, 2017.

\bibitem[Hoc17]{hochsterLEC}
M.~Hochster.
\newblock Math 615 Lecture Notes, 2017.
\newblock Available at
  \url{http://www.math.lsa.umich.edu/~hochster/615W17/615.pdf}.

\bibitem[Ive73]{iversen2006generic}
B.~Iversen.
\newblock {\em Generic local structure of the morphisms in commutative
  algebra}.
\newblock Lecture Notes in Mathematics, Vol. 310. Springer-Verlag, Berlin-New
  York, 1973.

\bibitem[KPR75]{KuPfRo75}
H.~Kurke, G.~Pfister, and M.~Roczen.
\newblock {\em Henselsche {R}inge und algebraische {G}eometrie}.
\newblock VEB Deutscher Verlag der Wissenschaften, Berlin, 1975.
\newblock Mathematische Monographien, Band II.

\bibitem[Laf65]{lafon1}
J.-P. Lafon.
\newblock S\'{e}ries formelles alg\'{e}briques.
\newblock {\em C. R. Acad. Sci. Paris}, 260:3238--3241, 1965.

\bibitem[Laf67]{lafon2}
J.-P. Lafon.
\newblock Anneaux hens\'{e}liens et th\'{e}or\`eme de pr\'{e}paration.
\newblock {\em C. R. Acad. Sci. Paris S\'{e}r. A-B}, 264:A1161--A1162, 1967.

\bibitem[Lan84]{Lang1984}
S.~Lang.
\newblock {\em Algebra}.
\newblock Addison-Wesley Publishing Company, Advanced Book Program, Reading,
  MA, second edition, 1984.

\bibitem[LT70]{LaTo70}
F.~Lazzeri and A.~Tognoli.
\newblock Alcune propriet\`a degli spazi algebrici.
\newblock {\em Ann. Scuola Norm. Sup. Pisa Cl. Sci. (3)}, 24:597--632, 1970.

\bibitem[Mat80]{matsumura1980commutative}
H.~Matsumura.
\newblock {\em Commutative algebra}, volume~56 of {\em Mathematics Lecture Note
  Series}.
\newblock Benjamin/Cummings Publishing Co., Inc., Reading, Mass., second
  edition, 1980.

\bibitem[Mil80]{milne1980etale}
J.~S. Milne.
\newblock {\em \'{E}tale cohomology}, volume~33 of {\em Princeton Mathematical
  Series}.
\newblock Princeton University Press, Princeton, N.J., 1980.

\bibitem[Mil13]{milneLEC}
J.~S. Milne.
\newblock Lectures on Etale Cohomology (v2.21), 2013.
\newblock Available at \url{www.jmilne.org/math/}.

\bibitem[Nag53]{Nagata1953}
M.~Nagata.
\newblock On the theory of {H}enselian rings.
\newblock {\em Nagoya Math. J.}, 5:45--57, 1953.

\bibitem[Nag54]{Nagata1954}
M.~Nagata.
\newblock On the theory of {H}enselian rings. {II}.
\newblock {\em Nagoya Math. J.}, 7:1--19, 1954.

\bibitem[Nag59]{Nagata1959}
M.~Nagata.
\newblock On the theory of {H}enselian rings. {III}.
\newblock {\em Mem. Coll. Sci. Univ. Kyoto Ser. A. Math.}, 32:93--101, 1959.

\bibitem[Nag62]{nagata1975local}
M.~Nagata.
\newblock {\em Local rings}.
\newblock Interscience Tracts in Pure and Applied Mathematics, No. 13.
  Interscience Publishers a division of John Wiley \& Sons\, New York-London,
  1962.

\bibitem[Pó22]{polya1922}
G.~Pólya.
\newblock Sur les séries entières, dont la somme est une fonction
  algébrique.
\newblock {\em L'Enseign. Math.}, pages 38--47, 1921--1922.

\bibitem[Ray70]{raynaud1970anneaux}
M.~Raynaud.
\newblock {\em Anneaux locaux hens\'{e}liens}.
\newblock Lecture Notes in Mathematics, Vol. 169. Springer-Verlag, Berlin-New
  York, 1970.

\bibitem[Rib85]{Ribenboim1985EquivalentFO}
P.~Ribenboim.
\newblock Equivalent forms of {H}ensel's lemma.
\newblock {\em Exposition. Math.}, 3(1):3--24, 1985.

\bibitem[Ron18]{Rond2018}
G.~Rond.
\newblock {Artin Approximation}.
\newblock {\em {Journal of Singularities}}, 17:108--192, 2018.
\newblock 108 pages.

\bibitem[Rui93]{Ruiz1993}
J.~M. Ruiz.
\newblock {\em The basic theory of power series}.
\newblock Advanced Lectures in Mathematics. Friedr. Vieweg \& Sohn,
  Braunschweig, 1993.

\bibitem[{Sta}20]{stacks-project}
T.~{Stacks Project Authors}.
\newblock \textit{Stacks Project}.
\newblock \url{https://stacks.math.columbia.edu}, 2020.

\bibitem[Str14]{straub2014}
A.~Straub.
\newblock Multivariate {A}p\'{e}ry numbers and supercongruences of rational
  functions.
\newblock {\em Algebra Number Theory}, 8(8):1985--2007, 2014.

\bibitem[SW88]{SharifWoodcock}
H.~Sharif and C.~F. Woodcock.
\newblock Algebraic functions over a field of positive characteristic and
  {H}adamard products.
\newblock {\em J. London Math. Soc. (2)}, 37(3):395--403, 1988.

\end{thebibliography}

\end{document}